\documentclass{amsart}
\usepackage{amssymb,amsmath,amsrefs}
\usepackage[colorlinks=true]{hyperref}
\hypersetup{urlcolor=blue, citecolor=green}
\usepackage[dvips]{graphicx}
\usepackage{xcolor}
\usepackage{comment}

\usepackage[raggedrightboxes]{ragged2e}

\newcommand{\SN}{\mathcal{N}}
\newcommand{\SE}{\mathcal{E}}
\newcommand{\SF}{\mathcal{F}}

\newcommand{\R}{\mathbb{R}}
\newcommand{\eps}{\varepsilon}
\newcommand{\SU}{\mathcal{U}}

\theoremstyle{remark}
\newtheorem{remark}{\bf{Remark}}[section]

\usepackage[pagewise]{lineno}

\usepackage{amssymb}
\usepackage{enumerate}

\theoremstyle{definition}
\newtheorem{theorem}{Theorem}[section]
\theoremstyle{definition}
\newtheorem{lemma}[theorem]{Lemma}
\theoremstyle{definition}
\newtheorem{definition}[theorem]{Definition}

\newtheorem{prop}{Proposition}[section]
\usepackage{soul,color}
\newtheorem{maintheorem}{Theorem}

\newtheorem{coro}[maintheorem]{Corollary}

\numberwithin{equation}{section}

\begin{document}

\title[Contributions to  the Theory of ASH Flows]{Contributions to  the Theory of Asymptotically Sectional Hyperbolic Flows}


\author{Alexander Arbieto}
\address{Instituto de Matem\'atica, Universidade Federal do Rio de Janeiro, Rio de Janeiro, Brazil.}
\email{arbieto@im.ufrj.br}

\author{Miguel Pineda}
\address{Instituto de Matem\'atica, Universidade Federal do Rio de Janeiro, Rio de Janeiro, Brazil.}
\email{miguel.pineda@im.ufrj.br}

\author{Elias Rego}
\address{Department of Applied Mathematics, AGH University of Science and Technology, Krakow, Poland.}
\email{rego@agh.edu.pl}

\author{Kendry Vivas}
\address{Departamento de Matematicas, Universidad Cat\'olica del Norte, Antofagasta, Chile.}
\email{kendry.vivas01@ucn.cl}

\thanks{AA was partially supported by CAPES, CNPq, PRONEX-Dynamical Systems and FAPERJ E-26/201.181/2022 ``Programa Cientista do Nosso Estado from Brazil''. MP was partially supported by CAPES from Brazil. ER was supported by the European Union's European Research Council Marie Sklodowska-Curie grant No.101151716}

\keywords{Asymptotically sectional-hyperbolic, growth rate of periodic orbits, entropy}
\subjclass[2020]{Primary: 37C10, Secondary: 37C27.}

\begin{abstract}
In this paper, we make several contributions to the theory of asymptotically sectional-hyperbolic (ASH) flows. First, we prove that every star ASH attractor for a $C^2$ vector field is, in fact, sectional-hyperbolic (SH). Second, we establish that all ASH attractors exhibit the property of entropy flexibility. Additionally, we show that any ASH attractor for three-dimensional vector fields is entropy-expansive and admits periodic orbits. Finally, we provide a lower bound for the growth rate of periodic orbits in an ASH attractor.

\end{abstract}

\maketitle

\tableofcontents

\section{Introduction}

The study of the dynamics of flows, also known as continuous-time dynamics, from both topological and statistical viewpoints, is a significant area of research in mathematics. This field traces its origins back to the work of Poincaré, who utilized these concepts to deepen the understanding of the topology of underlying manifolds.

Over time, the study of continuous-time dynamics evolved into a distinct mathematical discipline. Following the groundbreaking contributions of Anosov and Smale, it emerged as a particularly fruitful area of research. Specifically, the discovery of the hyperbolic nature of flows through Smale's horseshoe \cite{S} as a source of stability, along with Anosov's work on the hyperbolicity of the geodesic flow on negatively curved Riemannian manifolds \cite{A}, became celebrated cornerstones of the theory. These developments introduced techniques from both differential topology and ergodic theory, greatly advancing the understanding of the dynamics of a broad and important class of flows.

Subsequently, concepts from information theory and statistical physics, introduced by Kolmogorov, Sinai, Ruelle, Bowen, and others, were incorporated to better understand the complexity of dynamics and the relevant invariant measures of the system, particularly those that maximize entropy and, by the variational principle, achieve topological entropy. This integration led to a fruitful symbiosis between topological dynamics and the ergodic theory of dynamical systems.

One particular source of chaos and positive entropy is the presence of horseshoes. The pursuit of identifying horseshoe-type subdynamics within more general flows became a highly active area of research (\cite{Xma}, \cite{FL}, \cite{GY}). A more challenging question also emerged, which can be termed the \textit{flexibility of entropy}: given any value between zero and the 
topological entropy, can we find a compact subset (or an ergodic measure) whose entropy matches that value? One of the aims of this article is to address this question for certain flows that present new challenges due to the presence of singularities, as will be discussed in the following sections.

Although hyperbolic theory is very powerful, it requires that the subspace generated by the vector field is continuous,  which implies, among other things, that its dimension must be locally constant. Consequently, no singularity can be approached by regular orbits within the hyperbolic set.  In \cite{Lo}, Lorenz discovered a robust attractor (with a dense orbit) that exhibits some properties resembling those of hyperbolic systems, yet includes singularities that are accumulated by regular orbits within the attractor. Inspired by this model, Guckenheimer, Shilnikov and Turaev in an independently way (\cite{Gu1} and \cite{T}) introduced a geometric example that resembles the system studied by Lorenz. Therefore, to develop a comprehensive program for understanding most dynamical systems, it is essential to analyze such open sets.

The search for such a program motivated the search for a systematic theory to describe such dynamics, which was established by Morales and Metzger \cite{Me3} under the name \textit{sectional hyperbolicity} (see Section \ref{Prelim}). In their paper, they proved that the geometric Lorenz attractor is sectional-hyperbolic. Moreover, in a seminal paper by Morales, Pacifico, and Pujals \cite{MPP}, it was shown that any robust attractor in a 3-dimensional manifold must be sectional-hyperbolic. In fact, Tucker \cite{Tu} later proved that the attractor obtained from original Lorenz equations is, in fact, sectional-hyperbolic.

One might assume that sectional-hyperbolic systems share similar properties with hyperbolic ones. While this is sometimes true, it is important to exercise caution. For instance, the problem of finding a periodic orbit can yield different outcomes: any isolated nontrivial hyperbolic set must have a periodic orbit, but there are isolated nontrivial sectional-hyperbolic sets without periodic orbits \cite{Mo}. We will revisit this issue later. Another example is the following: Every isolated transitive hyperbolic set is robustly transitive, but this is not true in the sectional hyperbolic theory anymore

In his PhD thesis, Rovella constructed another flow, similar to the Lorenz attractor, but with a singularity exhibiting different behavior \cite{Ro}. He was able to find examples of attractors that, unlike the Lorenz attractor, do not exhibit robustness. However, inspired by several works on the quadratic family, such as those by Benedicks-Carleson \cite{BC} and Jakobson \cite{Ja}, as well as on homoclinic bifurcations (see also the work of Palis and Yoccoz \cite{PY}, he proved that such attractors persist in a certain sense, likewise in a codimension 2 submanifold. Once again, understanding such examples is crucial for a comprehensive understanding of dynamical systems.

It turns out that the Rovella attractor fits within another theory with a hyperbolic flavor: the \textit{asymptotic sectional-hyperbolic dynamics}, introduced by \cite{Mo4}. 
In fact, in \cite{SMV2} the authors shown that the Rovella attractor is an asymptotically sectional-hyperbolic attractor. Another example of sets satisfying this weak kind of hyperbolicity is known as the contractive singular horseshoe \cite{Mo4}. It should be noted that these examples are not sectional-hyperbolic.

In this article we continue the study of such dynamics. Let us now precise its definition. Let $M$ be a compact Riemannian manifold endowed with metric $d$, induced by the Riemannian metric $\Vert\cdot\Vert$. We denote by $X$ to a $C^1$ vector field on $M$, and we will refer by its {\it{flow}} on $M$ to the family of maps $\Phi=\lbrace X_t\rbrace_{t\in\mathbb{R}}$, induced by $X$. A compact subset  $\Lambda$ of $M$ is called $X$-invariant if $X_t(\Lambda)=\Lambda$, for every $t\in \mathbb{R}$. 

Next, recall that a compact invariant set $\Lambda$ has a {\it{dominated splitting}} if there are a continuous invariant  splitting $T_{\Lambda}M=E\oplus E^c$ (respect to $DX_t$) and constants $K,\lambda>0$ satisfying
\begin{displaymath}
\frac{\Vert DX_t(x)\vert_{E_x}\Vert}{m(DX_t(x)\vert_{E^c_x})}\leq Ke^{-\lambda t}, \quad \forall x\in\Lambda,\forall t>0,
\end{displaymath} 
where $m(A)$ denotes the conorm of $A$. In this case, we say that $E^c$ is {\it{dominated}} by $E$. We say that $\Lambda$ is \textit{partially hyperbolic} if $E$ is a contracting  subbundle , i.e.,$$\Vert DX_t(x)\vert_{E_x}\Vert\leq Ke^{-\lambda t},$$ for every $t>0$ and $x\in\Lambda$. Finally, denote by $W^s(Sing(X))$ the union of the stable manifolds of the singularities of $X$. 

\begin{definition}
Let $\Lambda$ be a compact invariant partially hyperbolic set of a vector field $X$. We say that $\Lambda$ is {\it{asymptotically sectional-hyperbolic}} (ASH for short) if the singularities of $\Lambda$ are hyperbolic and its central subbundle is eventually asymptotically expanding outside the stable manifolds of the singularities, i.e., there exists $C>0$ such that 
\begin{equation}\label{ash}
\limsup_{t\to+\infty}\frac{\log\vert \text{det}(DX_t(x)\vert_{L_x})\vert}{t}\geq C,
\end{equation}
for every $x\in \Lambda'=\Lambda\setminus W^s(Sing(X))$ and every two-dimensional subspace $L_x$ of $E^c_x$. We say that an ASH set is \textit{non-trivial} if it is not reduced to a singularity. 
\end{definition}

\begin{remark}
ASH sets satisfy the \textit{Hyperbolic lemma} i.e., any compact and invariant set without singularities is hyperbolic.  The proof of this result can be found in \cite{SMV2}. 
\end{remark}

It can be easily seen from the definition that the sectional-hyperbolic theory is encompassed within asymptotically sectional-hyperbolic theory, which in turn contains the hyperbolic theory. Moreover, the Rovella's attractor and the contracting singular horseshoes show that those inclusions are proper. An important consideration regarding the difference between sectional hyperbolicity and asymptotic sectional hyperbolicity is that in the case of sectional-hyperbolic dynamics, uniform estimates are often obtainable (which hold for nearby vector fields). However, in the asymptotic scenario, we must exercise more caution, as in the case of the Rovella attractor, where uniform estimates are not expected. Indeed, from ASH property we see that for every point $x\in\Lambda$ outside $W^s(Sing(X))$, and every plane $L_x \in G(2,F)$ (the  Grasmannian of two-planes contained in the subbundle $F$) there is an unbounded increasing sequence of positive numbers $t_k=t_k(x,L_x)>0$, called \textit{hyperbolic times}, such that
\begin{equation}\label{ht}
\vert \det DX_{t_k}(x)\vert_{L_x}\vert\geq e^{Ct_k},\quad k\geq 1.
\end{equation}
Any unbounded increasing sequence satisfying the relation \eqref{ht} will be called a sequence $C$-\textit{hyperbolic times} for $x$.

A vector field $X$ is \textit{star} if there is a $C^1$-neighborhood of $X$ formed by vector fields whose all singularities and periodic orbits are hyperbolic.  Star vector fields  form a key concept for dealing with global dynamics. They were introduced by Liao and Mañé, who showed that if the dynamics cannot bifurcate through non-hyperbolic periodic orbits, then robustly, all periodic orbits must exhibit uniform hyperbolic strength up to their period, along with a dominated splitting arising from the union of their hyperbolic splittings (see \cite{SLiao},\cite{Man}). Morales and Pacifico (in dimension 3) and Shi, Gan and Wen  (in dimension 4) proved that, generically, the star property is equivalent to sectional hyperbolicity (see \cite{MP5} and \cite{shiganwen}). However, in higher dimensions, this equivalence does not hold, as exemplified by the work of Bonatti and Da Luz \cite{BdL}. Thus, one can ask when the ASH theory diverges from the star theory. Our first result shows that, in any dimension, when restricted to asymptotic sectional-hyperbolic dynamics, the star property is equivalent to sectional hyperbolicity.

\begin{maintheorem}\label{StarinterASH}
    Every asymptotically sectional-hyperbolic attractor associated to a $C^2$-vector field $X$ on  $M$ satisfying the star property is sectional-hyperbolic.
    
\end{maintheorem}

The crucial step in proving prove Theorem \ref{StarinterASH} under $C^2$-regularity is the existence of a  periodic orbit contained in the attractor $\Lambda$. Nevertheless, by our next results, one can obtain periodic orbits under the assumption that $X$ is of class $C^1$ with positive entropy. Moreover, as we will see later in Theorem \ref{HomClass3}, this hypothesis is satisfied in the three-dimensional scenario. Consequently, Theorem \ref{StarinterASH} is valid for any ASH attractor with positive entropy associated to a $C^1$ vector field and, in particular, for any ASH attractor for $C^1$ vector fields on three-dinensional manifolds.

Next, we delve into the entropy theory of ASH flows. More precisely, we address their entropy flexibility. Here, we denote by $h(X)$ and $h_{\mu}(X)$ the topological and metric entropies of $X$, respectively (see Section \ref{Prelim} for precise definitions). We say that a subset $\Lambda$ has \textit{entropy flexibility} if for every $h\in [0,h(X))$ there are a compact  invariant subset $K\subset \Lambda$ and an invariant measure $\mu$ such that $$h=h(X|_K)=h_{\mu}(X).$$ 

In our next result, we address the entropy flexibility of ASH flows.

\begin{maintheorem}\label{entropyflex}
    Let $X$ be a $C^1$-vector field on $M$.  Suppose $M$ contains an asymptotically sectional-hyperbolic attractor $\Lambda$ for $X$. Then  $\Lambda$ has entropy fexibility.  
\end{maintheorem}

As mentioned earlier, due to the variational principle, it is natural to inquire whether a measure of maximal entropy exists. 
Bowen, in \cite{B}, introduced a property called \textit{entropy-expansiveness} to guarantee the upper semi-continuity of the entropy map, thus ensuring the existence of measures with maximal entropy. In our next result, we prove this property for ASH attractors in dimension three.

\begin{maintheorem}\label{Entrexp}
 Every asymptotically sectional-hyperbolic attractor $\Lambda$ associated with a $C^1$ vector field $X$ on a three-dimensional manifold $M$ is entropy-expansive. 
\end{maintheorem}

Another natural question concerns the positivity of entropy. As discussed earlier, this is related to the existence of horseshoe-like subdynamics, particularly the presence of periodic orbits. As showed by \cite{BM}, it is known that any attracting\footnote{A set $\Lambda$ is attracting if it has a neighborhood so that every point in the neighborhood eventually enters and remains within the set under the dynamics, i.e., $\Lambda=\bigcup_{t>0}X_t(U)$ for some open set satisfying $X_1(\overline{U})\subset U$.} sectional-hyperbolic set has a periodic orbit. However, there are attracting ASH sets without periodic orbits \cite{SMV2}. Thus, the issue of the existence of periodic orbits remains a subtle one. Our next result addresses this problem in dimension three.

\begin{maintheorem}\label{HomClass3}
Any asymptotically sectional-hyperbolic attractor $\Lambda$ associated to three-dimensional vector fields $X$ of class $C^1$ has a periodic orbit. Actually it contains a nontrivial homoclinic class. Thus its topological entropy is positive. If the periodic orbits are dense on $\Lambda$, then it is a homoclinic class.
\end{maintheorem}

We suspect that the attractor is generally a homoclinic class. Indeed, this holds true in the sectional-hyperbolic setting, as shown by Arroyo and Pujals \cite{AR}. However, if we consider higher regularity, this result extends to any dimension.

\begin{maintheorem}\label{HomClass}
    If a $C^2$-vector field $X$ on $M$ contains a asymptotically sectional-hyperbolic attractor, then $M$ contains a non-trivial homoclinic class.
\end{maintheorem}

 All the Theorems \ref{Entrexp},  \ref{HomClass3}  and \ref{HomClass} have consequences about the entropy of the attractor under perturbations.

\begin{coro}\label{lowersc}
Let $X$ be a $C^1$ vector field on $M$, and let $\Lambda$ be an asymptotically sectional-hyperbolic attractor. Then, there is a neighborhood $U$ of $\Lambda$ and a $C^1$ neighborhood $\mathcal{U}$ of $X$ such that $X\vert_{\Lambda}$ is a point of lower semicontinuity for the entropy function on $$\mathfrak{X}^1(M,U)=\left\lbrace Y\vert_{\Lambda_Y} : \Lambda_Y=\bigcap_{t\geq 0}Y_t(U)\right\rbrace.$$    In addition, if $M$ is three-dimensional, then $X|_{\Lambda}$ is a point of continuity for the entropy function on $\mathfrak{X}^1(M,U)$.
\end{coro}
    
In our final main result, we put the topological and measure theoretical viewpoint together to relate the entropy with the periodic orbits of a ASH flow. This result is based on the  work of Bowen  in the begining of seventies \cite{B}, where its shown that, for axiom A systems, there is a relationship between the topological entropy and the growth rate of the periodic orbits. Namely, he proved that if $f$ is an Axiom A diffeomorphism, then 
\begin{displaymath}
    \limsup_{t\to\infty}\frac{1}{t}\log(\# P_n(f))=h(f), 
\end{displaymath}
where $\# P_n(f)$ denotes the number of periodic orbits for $f$ of period $n$. Later, Katok in \cite{K} proved for $C^{1+\alpha}$ surface diffeomorphisms that the topological entropy is a lower bound for the growth rate of its periodic orbits. In \cite{WYZ} this result was extended to generic $C^1$ vector fields $X$. In this paper we obtain this result for ASH attractors associated to $C^1$ vector fields $X$. More precisely, we have the following theorem:  
\begin{maintheorem}\label{growth}
Let $X$ be a $C^1$ vector field on $M$.  Suppose $M$ contains an asymptotically sectional-hyperbolic attractor $\Lambda$ for $X$. Then,  
\begin{equation*}
    \limsup_{t \to \infty}\frac{1}{t}\log \# P_{t}(X|_{\Lambda}) \geq h(X|_{\Lambda}).
\end{equation*}
\end{maintheorem}

The organization of this paper is as follows: In Section \ref{Prelim}, we introduce the basic concepts and results that will be used throughout this text. In section \ref{SecHypMeasures}, we present some basic preliminary lemmas regarding the theory of ASH flows and we will study the hyperbolicity of the ergodic measures for ASH sets. Section \ref{SecEntrApp} is devoted to studying the entropy theory of ASH attractors and some of its applications. More precisely, we prove Theorem \ref{StarinterASH}, Theorem \ref{entropyflex}, Theorem \ref{HomClass}, Corollary \ref{lowersc} and Theorem \ref{growth}. In section \ref{3d} we will study ASH dynamics in the three-dimensional setting and we will provide the proofs for Theorem \ref{Entrexp} and Theorem \ref{HomClass3}.

\section{Preliminaries}\label{Prelim}

This section is devoted to provide the precise definitions for all the concepts used in this work. 
We also provide some useful known results from the theory of singular flows.

\subsection{Basic setting}
Throughout this work, $M$ denotes a compact Riemannian manifold endowed with a metric $d$, induced by the Riemannian metric $\Vert\cdot\Vert$. We denote by $X$ to a $C^r$ vector field on $M$.  For any $r\geq 1$,  denote by $\mathfrak{X}^r(M)$ the space of $C^r$ vector fields on $M$ endowed the $C^r$ norm. Denote by $\Phi=\lbrace X_t\rbrace_{t\in\mathbb{R}}$ the \textit{flow} induced by $X$. As usual, we say that a set $\Lambda$ is $X$-invariant if $X_t(\Lambda)=\Lambda$, for every $t\in \mathbb{R}$.

For $x\in M$, the \textit{orbit} of $x$ is the set  $$\mathcal{O}(x)=\{X_t(x) : t\in \mathbb{R}\}.$$ 
For $a,b\in\mathbb{R}$, the \textit{orbit segment from $a$ to $b$} of a point $x$ is defined by $$X_{[a,b]}(x)=\lbrace X_t(x) : t\in[a,b]\rbrace.$$ We say that $\sigma\in M$ is a \textit{singularity} of $X$ if $X(\sigma)=0$. Denote the set of singularities of the vector field $X$ by $Sing(X)$.  A point $x\in M$ is \textit{regular} if it is not a singularity. A regular point $x$ is \textit{periodic} if there is $t>0$ such that $X_t(x)=x$. The \textit{period} $\pi(x)$ a periodic point $x$ is defined as the smallest positive number $t$ satisfying $X_t(x)=x$. Denote by $Per(X)$ the set of periodic orbits of $X$. 
 
We say that a compact invariant set $\Lambda$ is {\it{attracting}} if there exists a neighborhood $U_0$ of $\Lambda$ (called trapping region) such that $$\overline{X_t(U_0)}\subset U_0,\quad\forall t>0,$$ and
\begin{displaymath}
\Lambda=\bigcap_{t\geq 0}X_t(U_0).
\end{displaymath}
We say that an attracting set $\Lambda$ is an  {\it{attractor}} if it is transitive, i.e., there is $z\in\Lambda$ such that $\overline{\mathcal{O}_+(z)}=\Lambda$, where $\mathcal{O}_+(z)$ denotes the positive orbit of $z$. 

\subsection{Hyperbolic, Sectional-Hyperbolic and Star Flows}

In the sequel, we shall recall some concepts that will be widely explored in this work, and were mentioned in the previous section. 
\begin{definition}[Dominated splitting]
  A compact invariant set $\Lambda$ is said to have a \textit{dominated splitting} if there are a continuous invariant  splitting $T_{\Lambda}M=E\oplus E^c$  and constants $K,\lambda>0$ satisfying:
\begin{displaymath}
   \frac{\Vert DX_t(x)\vert_{E_x}\Vert}{m(DX_t(x)\vert_{E^c_x})}\leq Ke^{-\lambda t}, \quad \forall x\in\Lambda,\forall t>0,
\end{displaymath}
where $m(A)$ denotes the conorm of a linear transformation $A$. In this case, we say that $E$ is \textit{dominated} by $E^c$. 
\end{definition}

\begin{definition}[Hyperbolic set]
A compact and invariant set $\Lambda$ is said to be \textit{hyperbolic} if it has a dominated splitting $T_{\Lambda} M=E^s\oplus \langle X \rangle \oplus E^u$ such that, $E^s$ is contracting  and $E^u$ is expanding, i.e., $$\Vert DX_t(x)\vert_{E_x}\Vert\leq Ke^{-\lambda t} \textrm{ and } \Vert DX_{-t}(x)\vert_{E^u_x}\Vert\leq Ke^{-\lambda t},$$ for every $t>0$ and $x\in\Lambda$.
\end{definition}    

A \textit{hyperbolic periodic point} is a periodic point whose orbit is a hyperbolic set.  A singularity is hyperbolic if it is a hyperbolic set.

\begin{definition}[Partially Hyperbolic set]
A compact and invariant set $\Lambda$ is said to be \textit{partially hyperbolic} if it has a dominated splitting $T_\Lambda M=E\oplus E^c$ such that $E$ is contracting.      
\end{definition}

\begin{definition}[Sectional Hyperbolic set]
Let $\Lambda$ be a compact invariant partially hyperbolic set of a vector field $X$ whose singularities are hyperbolic.  We say that $\Lambda$ is {\it{sectional-hyperbolic}} (SH for short) if its central subbundle is sectional expanding, i.e., there exists $K, \lambda>0$ such that for every two-dimensional subspace $L_x$ of $E^c_x$ one has
\begin{equation}\label{SH}
\vert \text{det}DX_t(x)\vert_{L_x}\vert\geq Ke^{\lambda t},\quad\forall x\in\Lambda,\forall t> 0.
\end{equation}
\end{definition}

We now conclude this subsection by providing the definition of star flow that will be considered in Theorem \ref{StarinterASH}.
\begin{definition}[Star flow]
    A vector field $X$ is said to be \textit{star vector field} (or \textit{star flow}) if there is a $C^1$-neighborhood $\SU$ of $X$ such that if $Y\in \mathcal{U}$, then any critical element of $Y$ is hyperbolic. We say that a compact invariant set $\Lambda$ for $X$ is \textit{star} if there is a neighborhood $U$ of $\Lambda$ such that $X$ is a star flow on $U$.
\end{definition}

\begin{remark} We have the following remarks
    \begin{itemize}
        \item The SH flows are star flows. Note that the partial hyperbolicity and the property (\ref{SH}) are open conditions. Therefore, by the hyperbolic lemma, we have that SH flows are star flows.
        \item ASH flows that are not SH are not contained within star flows. Indeed, by Theorem \ref{StarinterASH}, we have that the intersection of ASH flows and star flows are the SH flows. In particular, Rovella's attractor is an ASH flow that is not a star flow.
    \end{itemize}
\end{remark}

\subsection{Stable Manifolds and Homoclinic Classes}

Recall that for a hyperbolic periodic point $p$ associated to a $C^1$ vector field $X$, the \textit{strong stable} and  \textit{strong unstable manifold} of $p$ are defined, respectively, by 

$$W^{ss}(p)=\{y\in M :  \lim\limits_{t\to \infty}d(X_t(p),X_t(y))= 0\}$$
and 
$$W^{uu}(p)=\{y\in M :  \lim\limits_{t\to-\infty}d(X_t(p),X_t(y))= 0\}.$$
We then define the stable and unstable manifolds of $p$, respectively,  by $$W^s(p)=\bigcup_{t\in \mathbb{R}}W^{ss}(X_t(p))  \textrm{ and }  W^u(p)=\bigcup_{t\in \mathbb{R}}W^{uu}(X_t(p)).$$

Now, recall that for a hyperbolic periodic orbit $p$ associated to a $C^1$ vector field $X$, the \textit{homoclinic class} of $p$, denoted by $H(p)$, is defined as
\begin{displaymath}
H(p):=\overline{W^s(p)\pitchfork W^u(p)}.
\end{displaymath}
In this way, for any pair of periodic orbits $\gamma_p$ and $\gamma_q$, we say that  $\gamma_p\sim \gamma_q$  if $W^s(p)\pitchfork W^u(p)$ and $W^s(q)\pitchfork W^u(p)$. In this case, they belong to the same homoclinic class. We say that a homoclinic class $H(p)$ is \textit{non trivial} if $H(p)\neq\mathcal{O}(p)$.

\subsection{Ergodic Theory}
Next, we reall some concepts from the ergodic theory of flows. We say that a Borelian probility measure $\mu$ in $M$ is \textit{$\Phi$-invariant} if $\mu(X_t(A))=\mu(A)$, for every Borelian subset $A\subset M$ and every $t\in \mathbb{R}$. An invariant probability measure is \textit{ergodic} if for every $X$-invariant set 
$A\subset M$, one has $\mu(A)\mu(A^c)=0$.

The following theorem, due to Oseledets \cite{Osel}, will be crucial to our purposes.

\begin{theorem}[\textbf{Oseledets}] \label{Oseledecflow}
 Let $X$ be a $C^1$-vector field and $\mu$ a $\Phi$-invariant probability measure. There is a $\Phi$-invariant full set $B$ such that for every $x\in B$, there are $k(x)>1$, real numbers $\chi_1(x)<\cdots < \chi_{k(x)}(x)$, and a splitting $$T_xM=\bigoplus_{i=1}^{k(x)} H_i(x)$$ satisfying the following properties:

\begin{itemize}
\item The maps $x\mapsto H_i(x)$, $i=1,\ldots, k(x)$, are measurable. 
    \item The splitting is $DX_t$-invariant, i.e., $DX_t(x)H_i(x)=H_i(X_t(x))$ for every $t\in\mathbb{R}$. 
    \item For every $v\in H_i(x)\setminus\lbrace0\rbrace$, $i=1,\ldots, k(x)$), the following limit exists: 
    $$\chi(x,v)=\lim_{t \rightarrow \pm\infty} \frac{1}{\vert t\vert} \log \Vert DX_t(x) v\Vert=\chi_i(x).$$
    \item For $S\subset N:=\{1,2,\dots,k(x)\}$, let $H_S(x):=\bigoplus_{i\in S}H_i(x)$. Then, 
\begin{equation} \label{angles} 
    \lim_{t \rightarrow \infty} \frac{1}{t} \log (\sin\vert\measuredangle (H_S(X_t(x)),H_{N\setminus S}(X_t(x)))\vert)=0,
\end{equation}
where $\measuredangle (u,v)$ denotes  is the angle formed by $u$ and $v$.
    Moreover, for $u,v \in H_i$ we have
   \begin{equation*} \label{uvinH} 
    \lim_{t \rightarrow \infty} \frac{1}{t} \log \sin| \measuredangle (DX_t(x)u,DX_t(x)v)|=0.  
\end{equation*}
\end{itemize}
\end{theorem}

The numbers $\chi_i(x)$, $i=1,\ldots k(x)$, given in Theorem \ref{Oseledecflow} are called the \textit{Lyapunov exponents of }$x$.  Next, we provide a list of elementary facts about Lyapunov exponents that will be used in section 3.

\begin{itemize}
    \item $\chi(x,0)=-\infty$,
    \item $\chi(x,\alpha v)=\chi(x,v)$ for $\alpha \in \R \setminus \{ 0\}$,
    \item $\chi(x,v+w)\leq \max \{ \chi(x,v), \chi(x,w)\}$. Furthermore if $\chi(v)\neq \chi(w)$, then $\chi(x,v+w)= \max \{ \chi(x,v), \chi(x,w)\}$.  
\end{itemize}

\subsection{Topological Entropy}
Let $A\subset M $ and fix $\varepsilon, n>0$. We say that a subset $K$ of $A$ is 
\begin{itemize}
    \item a \textit{$n$-$\varepsilon$-separated set of $A$} if for any pair of distinct points $x,y\in K$ there is some $0\leq n_0\leq n$ such that $d(f^{n_0}(x),f^{n_0}(y))>\varepsilon$. Denote $S(n,\varepsilon, A)$ the maximal cardinality of an $n$-$\varepsilon$-separated subset of $A$.
    \item a \textit{$n$-$\varepsilon$-generator for $A$} if for every $x\in A$ there exists $y\in K$ such that $d(f^i(x),f^i(y))\leq \varepsilon$, for  every $0\leq i\leq n$. Denote $R(n,\varepsilon, A)$ the minimum carnality of the $n$-$\varepsilon$-generators for $A$.
\end{itemize}
 
Note that $ S(n,\varepsilon, A)$ and $R(n,\varepsilon,A)$ are always finite due to the compactness of $M$.  We then define the \textit{topological entropy of $f$ on $A$} as the number  $$h(f,A)=\lim\limits_{\varepsilon\to 0}\limsup\limits_{n\to\infty}\frac{1}{n}\log(S(n,\varepsilon, A))=\lim\limits_{\varepsilon\to 0}\limsup\limits_{n\to\infty}\frac{1}{n}\log(R(n,\varepsilon, A)).$$
Observe that when $A=M$, the number $h(f, A)=h(f)$ is the topological entropy of $f$. The definition of topological entropy for flows is similar. In fact, $h(\Phi,A)=h(X_1,A)$.

Now, for a map $f:M\to M$, the \textit{metric entropy} of an invariant measure $\mu$ for $f$ is given by 
\begin{displaymath}
    h_{\mu}(f)=\sup\lbrace h_{\mu}(f,\mathcal{P}): \mathcal{P}\text{ is a finite partition}\rbrace, 
\end{displaymath}
where \begin{displaymath}
    h_{\mu}(f,\mathcal{P})=-\lim_{n\to\infty}\frac{1}{n}\sum_{P\in\mathcal{P}_n}\mu(P)\log\mu(P),\quad \mathcal{P}_n=\mathcal{P}_0\vee f^{-1}(\mathcal{P}_1)\vee\cdots\vee f^{n-1}(\mathcal{P}).
\end{displaymath}

Denote by $\mathcal{M}(\Phi)$ the set of invariant measures for the flow of $X$. In this way, the metric entropy of an invariant measure $\mu$ for the flow $\Phi$ is defined as $h_{\mu}(\Phi)=h_{\mu}(X_1)$.  
 
On the other hand, we called \textit{potential} to a continuous function $\phi:M\to\mathbb{R}$. The \textit{topological pressure} of $\Phi$ with respect to $\phi$ is defined by 
\begin{displaymath}
    P(\Phi,\phi)=\sup_{\nu\in \mathcal{M}(\Phi)}\left\lbrace h_{\nu}(\Phi)+\int\phi d\nu\right\rbrace. 
\end{displaymath}
A measure $\mu\in\mathcal{M}(\Phi)$ is called an \textit{equilibrium state for the potential $\phi$ associated to $\Phi$} if the above supremum is attained at $\nu=\mu$. If $\phi\equiv 0$, the equilibrium state $\mu$ associated to $\phi$ is called \textit{measure of maximal entropy}.

Let $(M,d)$ be a compact metric space and let $f: M\to M$  be a homeomorphism. Recall that the  $\delta$-dynamical ball of a point $x\in M$ is the set 
 $$B^{\infty}_{\delta}(x,f)=\{y\in M : d(f^n(x),f^n(y))\leq \delta, \forall n\in \mathbb{Z}\}.  $$

\begin{definition}
A continuous flow $\Phi$ on a compact metric space $M$  is said to be \textit{entropy-expansive} if its time-one map is entropy-expansive,  i.e., there exists $\delta>0$ such that $h(X_1,B^{\infty}_{\delta}(x, X_1))=0$, for every $x\in M$. 
\end{definition}

In \cite{PYY} the authors showed that any SH invariant set for $C^1$ flows is entropy-expansive in any dimension. The Theorem \ref{Entrexp} states that this property holds for three-dimensional ASH attractors. 

It is worth mentioning that we will prove our results by different approaches. As we will see in Section \ref{3d}, this is due to the lack of uniform area expansion in the central bundle. This represents the main challenge when extending results from the sectional-hyperbolic to the ASH setting. This motivates us to develop new tools to deal with these issues.

\section{Hyperbolic measures for ASH flows}\label{SecHypMeasures}

In this section, we will derive Theorem \ref{erghyp}, an essential tool for proving our main results. To present the statement of this result, it is necessary to recall the notion of a hyperbolic measure. 
\begin{definition}
 An invariant probability measure $\mu$ for the flow of a $C^1$ vector field $X$ is a \textit{hyperbolic measure} if $\mu$-almost every point has only one zero Lyapunov exponent, which corresponds to the flow direction.   
\end{definition} 

Recall that an invariant measure $\mu$ for the flow of $X$ is \textit{supported on $\Lambda$} is $supp(\mu)\subset\Lambda$, where $supp(\mu)$ denotes the support of $\mu$. The main result of this section is as follows. 

\begin{theorem} \label{erghyp}
Let $\Lambda$ be an asymptotically sectional-hyperbolic set for a $C^1$ vector field $X$, and let $\mu$ be a regular ergodic invariant measure for the flow of $X$ supported on $\Lambda$. Then, $\mu $ is a hyperbolic measure.
\end{theorem}

Before proving Theorem \ref{erghyp}, let us collect some some previous lemmas. We decided to present their proofs for the sake of completeness. 

\begin{lemma} \label{X(x)inE^c}
    Let $\Lambda$ be a compact invariant set with a partially hyperbolic splitting  $T_{\Lambda}M=E\oplus E^c$. Then, $X(x)\in E^c(x)$ for any point $x \in \Lambda.$
\end{lemma}

\begin{proof}
First, if $x\in \text{Sing}(X)$, is clear that $X(x)=0\in E_x^c$. Let $x\in\Lambda\setminus Sing(X)$. We claim that $X(x)\notin E_x^s$. Indeed, assume that there is $x_0\in\Lambda\setminus Sing(X)$ such that $X(x_0)\in E_{x_0}^s$. Since $E^s$ is invariant, we have that
\begin{displaymath}
X(X_t(x_0))=DX_t(X(x_0))\in E_{X_t(x_0)}^s,\quad\forall t\in\mathbb{R}.
\end{displaymath}
So, for every $y\in\alpha(x_0)$ one has by continuity of $E^s$ and $X$ that 
\begin{displaymath}
X(y)=\lim_{k\to\infty}X(X_{-t_k}(x_0))\in E_y^s,
\end{displaymath}
so that 
\begin{displaymath}
\Vert X(X_t(y))\Vert=\Vert DX_t(X(y))\Vert\leq e^{-\lambda t}\to0, t\to+\infty. 
\end{displaymath}
Thus, $\omega(y)\subset Sing(X)$ for every $y\in\alpha(x_0)$. Since $\alpha(x_0)$ is closed and invariant it follows that $Sing(X)\cap\alpha(x_0)\neq\emptyset$. 

\textbf{Claim: }Any $\sigma\in Sing(X)\cap\alpha(x_0)$ is of saddle type.

 Indeed, if $\sigma$ were a contracting singularity, it follows that  $\sigma$ is a repeller for  $-X$, which contradicts that $\sigma\in\alpha(x_0)$; while if $\sigma$ were a repelling singularity, we would obtain $E_{\sigma}^s=\lbrace 0\rbrace$, whic is impossible since $\text{dim}E^s$ is locally constant in a neighborhood of $\sigma$ and $X(X_{-t_k}(x_0))\in E_{X_{-t_k}(x_0)}^s$. 

So, we have the following cases for $\alpha(x_0)$:
\begin{itemize}
\item $\alpha(x_0)=\lbrace\sigma\rbrace$: In this case, $x_0\in W^u(\sigma)$. Let consider the vector
\begin{displaymath}
v_t=\frac{X(X_t(x_0))}{\Vert X(X_t(x_0))\Vert},\quad t\in\mathbb{R}.
\end{displaymath}   
Clearly, $v_t\in T_{x_0}W^u(\sigma)\cap E_{X_t(x_0)}^s$, thus
\begin{displaymath}
v=\lim_{t\to-\infty}v_t\in T_{\sigma}W^u(\sigma)\cap E_{\sigma}^s=E_{\sigma}^u\cap E_{\sigma}^s\text{ and }\Vert v\Vert=1.
\end{displaymath}
This shows that $v\neq 0$, which is absurd.
\item $\alpha(x_0)\neq\lbrace\sigma\rbrace$: In this case there is $z\in \left(W^u(\sigma)\setminus\lbrace\sigma\rbrace\right)\cap\alpha(x_0)$ because $\sigma$ is of saddle type. Hence, we obtain a contradiction by reasoning in an analogous way to the above case, by considering $z$ instead of $x_0$. 
\end{itemize} 
This proves the claim. 

Now, notice by the above claim that for every regular popint $x\in\Lambda$, the angle between $X(x)$ and $E_x^s$, denoted by $\measuredangle(X(x),E_x^s)$, is positive. Then, if $V$ is a small neighborhood of the singularities contained in $\Lambda$, the continuous function $\measuredangle(X(\cdot),E_{(\cdot)}^s):\Lambda\setminus V\to\mathbb{R}^+$ is uniformly positive, i.e., there is $\theta>0$ such that $\measuredangle(X(x),E_{x}^s)\geq\theta$ for any $x\in\Lambda\setminus V$. 

Let $x\in\Lambda\setminus\text{Sing}(X)$. Assume that $x\notin W^u(\sigma)$ for every $\sigma\in \text{Sing}(X)\cap\Lambda$. Then, there is a regular point $y\in\alpha(x)$. In particular, we can choose a neighborhood $V_y$ of $y$ satisfying $\measuredangle(X(z),E_{z}^s)\geq\theta$, $\theta>0$, for any $z\in V_y$. In this way, if the sequence $t_n\to+\infty$ satisfies $X_{-t_n}(x)\to y$, it follows that $\measuredangle(X(X_{-t_n}(x)),E_{X_{-t_n}(x)}^s)\geq\theta$ for $n$ large enough. So, since the splitting is dominated, we have 
\begin{displaymath}
\measuredangle(DX_t(X(X_{-t_n}(x))),DX_t(E_{X_{-t_n}(x)}^c))\to0,\quad n\to+\infty.
\end{displaymath}   
In particular, for $\varepsilon>0$ there is $n\in\mathbb{N}$ for which
\begin{displaymath}
\measuredangle(X(x),E_{x}^c)=\measuredangle(DX_{t_n}(X(X_{-t_n}(x))),DX_{t_n}(E_{X_{-t_n}(x)}^c))<\varepsilon. 
\end{displaymath} 
Therefore, $\measuredangle(X(x),E_{x}^c)=0$, which shows thgat $X(x)\in E_x^c$. 

On the other hand, assume that $x\in W^u(\sigma)$ for some $\sigma\in Sing(x)$. Since $T_{\sigma}W^u(\sigma)\cap E_{\sigma}^s=\lbrace 0\rbrace$, it follows that $T_xW^u(\sigma)$ is in the complement of $E_x^s$, so that $T_xW^u(\sigma)\subset E_x^c$. Then, $X(x)\in E_x^c$, because $X(x)\in T_xW^u(\sigma)$. This proves the result. 
\end{proof}

Now, recall that an ergodic probability measure for $X$ is \textit{atomic} if it is supported on either a singularity or periodic orbit. We say that a probability measure $\mu$ is \textit{regular} if it is not atomic.
Let us begin with the following lemma, which says that a regular invariant measure for a $C^1$ vector field $X$ is not supported on $W^s(Sing(x))$.

\begin{lemma}\label{muest}
    Let $\Lambda$ be a compact invariant set for a $C^1$-vector field $X$. If $\mu$ is a regular ergodic measure for the flow of $X$, then $\mu(W^s(Sing(X))=0$.
\end{lemma}
\begin{proof}
     Since $W^s(Sing(X))\setminus Sing(x)$ is invariant, the ergodicity of $\mu$ implies that $\mu(W^s(Sing(X))\setminus Sing(x))$ has either full or zero measure. However, since no point in $W^s(Sing(X))\setminus Sing(x)$ can be recurrent, we concliude by Poincar\'e's recurrence theorem that $\mu(W^s(Sing(X))\setminus Sing(X))=0$. Finally, since $\mu$ is regular and $Sing(X)$ is finite, we have $\mu(Sing(X))=0$. Therefore, $\mu(W^s(Sing(X)))=0$. 
\end{proof}

Now, we present the next lemma, which separates negative and non-negative exponents within the ASH  splitting. 

\begin{lemma} \label{E-E^c+}
Let $\Lambda$ be an asymptotically sectional-hyperbolic set for a $C^1$ vector field $X$, and let $\mu$ be a regular ergodic invariant measure for $X$ supported on $\Lambda$. Let $T_{\Lambda}M=E \oplus E^c$ be the ASH splitting of $\Lambda$. Then,
$$E(x)=\bigoplus_{\chi(x)<0} H_{i}(x) \text{ and } E^c(x)=\bigoplus_{\chi(x)\geq0} H_{i}(x),$$
for $\mu$-almost every point $x\in\Lambda$, where the the subspaces $H_i(x)$ are given by Theorem \ref{Oseledecflow}. 
\end{lemma}

\begin{proof}
    Let $\mu$ be a regular and ergodic measure for $X$ supported on $\Lambda$ and let $B\subset\Lambda$ be the set given by Theorem \ref{Oseledecflow}.
Denote 
$$F^-(x)=\bigoplus_{\chi(x)<0} H_{i}(x) \text{ and } F^+(x)=\bigoplus_{\chi(x)>0} H_{i}(x),\quad  \forall x\in B.$$

We claim that $$E(x)=F^-(x) \textrm{ and } E^c(x)=\langle X(x)\rangle \oplus F^+(x)\quad \forall x\in B.$$ 

To see why the claim holds, recall that $\mu$ is a regular measure and, therefore, is not supported on $Sing(X)$. Moreover, by Lemma \ref{muest} we can assume that $B\cap W^s(Sing(X))=\emptyset$. Let us  first show that $E(x)=F^-(x)$. Indeed, due to the uniform contraction of $E$, the Lyapunov exponent of $x$, concerning any $v\in E(x)$, is negative and therefore $E(x)\subset F^-(x)$. Now, suppose there is a non-zero vector $v\in F^-(x)\setminus E(x)$. We have that $v=u+w$ with $u \in E(x)$, $w\in E^c(x)$ and $w\neq0$. Notice that $w\notin\langle X(x)\rangle$, because
$$\chi(x,w)=\chi(x,v-u)\leq \max\{\chi(x,v),\chi(x,-u)\}\leq \max\{\chi(x,v),\chi(x,u)\}<0.$$
Moreover, we can assume that $w$ satisfies $w\perp X(x)$ and $\Vert w\Vert=1$.

Recall that $X(x)\in E^c(x)$ by Lemma \ref{X(x)inE^c}, and denote $L_x=\langle X(x), w\rangle$ the two-dimensional subspace spanned by the vectors $X(x)$ and $w$. Note that $L_x$ is contained in $E^c(x)$. Let $\theta_t(x)$ be the angle between $DX_t(x)X(x)=X(X_t(x))$ and $DX_t(x)w$. Then,
\begin{eqnarray} \label{AreaPara}
 | \det(DX_t(x)|_L) |=\frac{\sin \theta_t\cdot\Vert X(X_t(x))\Vert\cdot\Vert DX_t(x)w\Vert}{\Vert X(x)\Vert}.  
\end{eqnarray}
On one hand, by the choice of $B$ and the asymptotic area expansion of $\Lambda$, we have
\begin{equation} \label{expASH}
\limsup_{t \rightarrow \infty} \frac{1}{t} \log \Vert \det DX_{t}(x)|_{L_x}\Vert>C>0.
\end{equation}
On the other hand, by definition of the Lyapunov exponents,
\begin{equation} \label{somaexp}
  \limsup_{t \rightarrow \infty} \frac{1}{t} \log \Vert\det DX_{t}(x)|_{L_x}\Vert= \chi(x,X(x))+\chi(x,w)<0.
\end{equation}
This is a contradiction. Therefore, $E(x)=F^-(x)$. In particular, 
\begin{equation} \label{E=F}
  \dim(E^c(x))=\dim(\langle X(x)\rangle\oplus F^+(x)).  
\end{equation}

Next, we show that $E^c(x)=\langle X(x)\rangle\oplus F^+(x)$. For $u \in E^c(x)\cap \langle X(x)\rangle^{\bot}$, by repeating the above argument for the plane $M=\langle X(x),u \rangle$ we have that $\chi(x,u)>0$. This implies that $u \in F^+(x)$. Therefore, $$E^c(x)=(E^c(x) \cap \langle X(x)\rangle^{\bot}) \oplus \langle X(x)\rangle \subset F^+(x) \oplus \langle X(x)\rangle,$$ 
and by (\ref{E=F}) we have $E^c(x)=\langle X(x)\rangle  \oplus F^+(x)$.
\end{proof}

Finally, we are ready to prove Theorem \ref{erghyp}.

\begin{proof}[Proof of Theorem \ref{erghyp}]

Let $\mu$ be a regular ergodic measure for $X$ and let $B$ be the set given by Oseledets theorem. By Lemma \ref{muest}, we can assume $B\cap W^s(Sing(X))=0$. Next, let $T_\Lambda M=E\oplus E^c$ be the ASH splitting associated to $\Lambda$. For $x\in B$, let $H(x)$ be vector subspace, given by Oseledets theorem, associated with the null Lyapunov exponents. Notice that by Lemma \ref{E-E^c+} one has $\langle X(x)\rangle \subset H(x)\subset E_x^c=\langle X(x)\rangle  \oplus F^+(x)$, where $F^+(x)$ is the subspace associated to the positive Lyapunov exponents. So, $H(x)=\langle X(x)\rangle$. This proves the result.  
\end{proof}

\section{Entropy of ASH attractors and applications}\label{SecEntrApp}

In this section, we begin to prove our main theorems. Here, we will be specifically interested in the entropy properties of ASH sets on manifolds of any dimension. In addition, we shall apply these  to provide a proof for Theorem \ref{StarinterASH}, which relates the theories of ASH, SH and star flows.

\subsection{Entropy Flexibility}
Let us begin by studying the entropy flexibility of ASH attractors. Next, we shall be concerned with the proof of Theorem \ref{entropyflex}, whose 
 one of the main ingredients is the Theorem \ref{horseshoe}. Let us first recall the concept of  \textit{horseshoe} for flows. Let $M$ be a compact metric space and  $f: M\to M$ be a homeomorphism. Let $\rho: M\to \mathbb{R}$ be a continuous function. The function $\rho$ is called the roof function. Define  $\overline{M}$  as the quotient space of $M\times \mathbb{R}$, through the equivalence relation $$(x,\rho(x))\sim (f(x),0).$$ It is well known that $\overline{M}$ can be endowed with a metric, making it a compact metric space, we refer the reader to \cite{BW} for the precise details. 

\begin{definition}\label{Def-Susp}
    The suspension flow of $f$ with roof function $\rho$ is the flow $\phi: \mathbb{R}\times \overline{M}\to \overline {M}$ defined 
  as $\phi(t,(x,s))=(f^n(x),s')$, where $n$ and $s'$ satisfy
$$\sum_{j=1}^{n-1}\rho(f^j(x))+s'=t+s, \,\,\,\, \,\,\, 0\leq s'\leq \rho(f^n(x)).$$
\end{definition}

\begin{definition}
    A compact invariant set $K\subset M$ for the flow of $X$ is a \textit{horseshoe} if there is a roof function $\rho$ such that $\Phi|_{K}$ is conjugated to the suspension of the $k$-symbol shift map $\sigma:\Sigma_k\to \Sigma_k$ with roof $\rho$.
\end{definition}
 A \textit{hyperbolic horseshoe} is a horseshoe admitting a hyperbolic splitting for the tangent flow.      
\begin{lemma}\label{horseshoe}
    Let $\Lambda$ be an asymptotically sectional-hyperbolic attractor for a $C^1$-vector field $X$. Suppose $\mu$ is an ergodic measure for the flow of $X$, supported on $\Lambda$ with positive entropy. Then, for every $\eps>0$, there is a hyperbolic horseshoe $K_{\eps}\subset \Lambda$ such that $$|h(X|_{K_{\eps}})-h_{\mu}(X)|\leq\eps.$$
\end{lemma}

The previous result was previously obtained for star flows in \cite{LSWW}, but its proof can be completely reproduced in the context of ASH attractors except for one ingredient, which was unknown until the present work. Namely, the hyperbolicity of ergodic regular measures. Fortunately, this was achieved in Theorem \ref{erghyp}. In this way, we omit the proof of Lemma  \ref{horseshoe} to avoid unnecessary repetition of the same arguments. Now, we are ready to prove Theorem \ref{entropyflex}.

\begin{proof}[Proof of Theorem \ref{entropyflex}]
First, notice that if $h(X|_{\Lambda})=0$, then the result is trivial. So, suppose $h(X|_{\Lambda})=h>0$. By the Variational Principle, there is a sequence of ergodic measures $\mu_n$, supported on $\Lambda$ satisfying $h_{\mu_n}(X\vert_{\Lambda})\to h$. In particular, we can assume  $h_{\mu_n}(X\vert_{\Lambda})>0$, for every $n\geq 1$. In this way, $\mu_n$ is regular, for every $n\geq 1$. Then, Theorem \ref{erghyp} implies that every $\mu_n$ is hyperbolic. 

Fix $\mu_n$ and $\eps>0$. By Lemma \ref{horseshoe}, there is a hyperbolic horseshoe $K_{\eps,n}\subset \Lambda$ whose entropy is $\eps$-close to $h_{\mu_n}(X\vert_{\Lambda})$. So, \cite{LSWW}*{Proposition 1.2} implies that for every $a\in[0,h(X\vert_{K_{\eps,n}}))$, there is an ergodic measure $\mu_{a,n}$, supported on $K_{\eps,n}$, such that $h_{\mu_{a,n}}(X)=a$. 

Next, we will find a compact invariant subset  $K_{a,n}\subset K_{\eps,n}$ satisfying $h(X|_{K_{a,n}})=a$. First, since $K_{\eps,n}$ is a horseshoe, it is conjugated to the suspension flow $\Psi$ of a full shift map $\sigma:\Sigma \to \Sigma$ for some roof function $\rho: M\to (0,+\infty)$. By \cite{Walters}*{Section 7.3}, for every $b\in [0,h(\sigma))$, there is a compact and $\sigma$-invariant set $\Sigma_b\subset \Sigma$ such that $h(\sigma|_{\Sigma_b})=b$. 

We claim that the same property holds for $K_{\eps,n}$ and $\Phi\vert_{K_{\varepsilon,n}}$. Indeed, notice that by Abramov's formula \cite{Ab} one has $$h(X|_{K_{\eps,n}})=\sup\left\{\frac{h_{\mu}(\sigma)}{\int \rho d\mu} : \mu \textrm{ is } \sigma\textrm{-invariant}\right\}.$$
On the other hand, by definition of horseshoe one has $h(X|_{K_{\eps,n}})=h(\sigma)$, and therefore we have by the Variational Principle
$$h(X|_{K_{\eps,n}})=h(\sigma)=\sup\{h_{\mu}(\sigma) : \mu \textrm{ is } \sigma\textrm{-invariant}\}.$$ This shows that 
\begin{equation}\label{4.1}
    \int \rho d\mu=1 \text{ for every } \sigma\text{-invariant measure } \mu. 
\end{equation} 

Now fix a number $a\in [0,h(X|_{K_{\eps}}))$, and let $\Sigma_a\subset \Sigma$ such that $a=h(\sigma|_{\Sigma_a})$. Let $K_{a,n}$ be the suspension of $\sigma|_{\Sigma_a}$ with roof $\rho|_{\Sigma_a}$. Then, by Abramovs's formula and \eqref{4.1} we obtain 
\begin{eqnarray*}
    h(X|_{K_{a,n}})&=&\sup\left\{\frac{h_{\mu}(\sigma)}{\int \rho d\mu} : \mu \textrm{ is } \sigma\vert_{\Sigma_a}\textrm{-invariant}\right\} \\
    &=&\sup\{h_{\mu}(\sigma) : \mu \textrm{ is } \sigma\vert_{\Sigma_a}\textrm{-invariant}\} \\
    &=&h(\sigma|_{\Sigma_a})=a. 
\end{eqnarray*}
This proves the claim.

To finish the proof of entropy flexibility, notice that since $h_{\mu_n}(X)\to h$, then for every $a\in [0,h(X|_{\Lambda}))$, there is $\mu_n$ such that $h_{\mu_n}(X)>a$ and therefore, there are  $\mu_{a,n}$ and $K_{a,n}\subset \Lambda$ such that $h_{\mu_{a,n}}(X)=h(X|_{K_{a,n}})=a$.
\end{proof}

\subsection{Existence of Periodic Orbits}
Next, we shall present the proof of Theorem \ref{HomClass}. Before entering the proof's details, we need the following result:

\begin{lemma}\label{posentr}
    Let $X$ be a $C^2$ vector field over $M$. If $M$ contains an asymptotically sectional-hyperbolic attractor, then $h(X)>0$.
\end{lemma}
\begin{proof}
    Let $\Lambda$ be an ASH attractor for a $C^2$ vector field $X$ on $M$. By Theorem \cite{ASS}*{Theorem B}, $\Lambda$ admits a unique physical measure. Moreover, by the proof of \cite{SMV3}*{Lemma 3.3} this measure is also non atomic. So, the conclusion is derived by applying \cite{SMV}*{Theorem 2.4}, which states that any ASH attractor admitting a non-atomic physical measure has positive topological entropy.  
\end{proof}

\begin{proof}[Proof of Theorem \ref{HomClass}]
Assume that $M$ contains an asymptotically sectional hyperbolic attractor $\Lambda$. By Lemma \ref{posentr}, one has $h(X)>0$. So, Lemma \ref{horseshoe} implies that $\Lambda$ contains a hyperbolic horseshoe. Since hyperbolic horseshoes are homoclinic classes, the proof is complete. 
\end{proof}

Next, we present the proof of Corollary \ref{lowersc}.

\begin{proof}[Proof of Corollary \ref{lowersc}]
Let $\Lambda$ be an ASH attractor for a $C^1$-vector field $X$ on $M$. First, notice that if $h(X|_{\Lambda})=0$, then the result holds trivially. Indeed, since zero is the lowest possible value for the entropy of a flow, then for any other $C^1$-vector field  $Y$ on $M$ and for any compact and $Y$-invariant subset $\Lambda_Y\subset M$, one has $h(Y|_{\Lambda_Y})\geq h(X|_{\Lambda})=0$. 

Next, suppose $h(X|_{\Lambda})>0$ and let $U\subset M$ be a trapping region of $\Lambda$. Fix $\eps>0$. By Lemma \ref{horseshoe}  $\Lambda$ contains a horseshoe $K_{\eps}$ such that 
\begin{equation}\label{6}
h(X|_{K_{\eps}})\geq h(X\vert_{\Lambda})-\frac{\eps}{2}.    
\end{equation}

On the other hand, since hyperbolicity is a robust property, there is an open $\SU\subset \mathfrak{X}^1(M)$ such that any vector field $Y\in \SU$ admits a hyperbolic set $\Lambda_{Y,\eps}\subset\Lambda_Y\subset U$, which is a continuation of $K_{\eps}$, such that
\begin{equation}\label{7}
    h(Y_{\Lambda_Y})\geq h(Y\vert_{\Lambda_{Y,\varepsilon}})>h(X|_{K_{\eps}})-\frac{\eps}{2}.
\end{equation}

Now, we conclude from \eqref{6} and \eqref{7} that $h(Y_{\Lambda_Y})\geq h(X\vert_{\Lambda})-\eps$, and the proof is complete.   
\end{proof}

\subsection{Growth of Periodic Orbits} 
Now, we are going to prove Theorem \ref{growth}. To achieve our goal,  we need to use the technology of the linear Poincar\'e flow, a concept that we now recall in detail. For $X \in \mathfrak{X}^1(M)$, denote the \textit{normal bundle of $X$} by
$$\SN= \bigcup_{x\in M\setminus Sing(X)}\SN_x,$$
where $$\SN_x=\{ v \in T_xM : v\bot X(x)\}.$$ 
Denote  $O_x:T_xM \to \mathcal{N}_x$ the orthogonal projection on $\mathcal{N}_{x}$.

\begin{definition}
The \textit{linear Poincar\'e flow} associated to $X$ is the flow $\mathcal{P}=\lbrace P_t\rbrace_{t\in\mathbb{R}}$ on $\SN$ defined by $$ P_t(x)v:=O_{X_t(x)}DX_t(x)v,\text{ for all } t\in\mathbb{R}.$$   
\end{definition}

Let $T_{\Lambda}M=E\oplus E^c$ be an $X$-invariant splitting on $\Lambda$. Denote $$\SE_x=O_x(E(x))   \textrm{ and } \SE^c_x=O_x(E^c(x)).$$ In this case, one has a splitting of the normal bundle given by $\SN=\SE\oplus\SE^c$. Notice that if $\mu$ is a hyperbolic invariant measure for the flow of $X$, it is possible to consider the following splitting on $sup(\mu)$: $T_{supp(\mu)}M=E^s(x)\oplus\langle X(x)\rangle\oplus E^u(x)$, where 
$$E^s(x)=\bigoplus_{\chi_{i}(x)<0}H_i(x) \text{  and  } E^u(x)=\bigoplus_{\chi_i(x)>0}H_i(x).$$ 
Denote $\SE^s_x=O_x(E^s(x))$ and $\SE^u_x=O_x(E^u(x))$ for every $x\in supp(\mu)$. 

\begin{lemma}[\textit{Theorem 3.4} in \cite{WYZ}] \label{PerEnt}
Let $\mu$ be a regular hyperbolic ergodic measure of $X\in\mathfrak{X}^1(M)$. If the splitting $\mathcal{N}=\mathcal{E}^s\oplus \mathcal{E}^u$ is dominated for the linear Poincar\'e flow $\mathcal{P}$, then 
\begin{equation*}
    \limsup_{t \to \infty}\frac{1}{t}\log \# P_{t}(X) \geq h_{\mu}(X):=h_{\mu}(X_1).
\end{equation*}
\end{lemma}

In order to apply Lemma \ref{PerEnt} for proving Theorem \ref{growth}, we must to lift dominated splittings for $X$ to dominated splittings for the linear Poincaré flow. 

\begin{lemma} \label{DDom}
If $T_{\Lambda}M=E \oplus F$ is a dominated splitting, then $\mathcal{N}=\mathcal{E}\oplus \mathcal{F}$ is dominated splitting w.r.t. the linear Poincaré flow $P_{t}$, where $\SE_x=O_x(E(x)) $ and $\SF_x=O_x(F(x))$.     
\end{lemma}

\begin{proof}
   Fix  $v \in E$. Since  $$T_xM=\mathcal{N}_x \oplus \langle X(x)\rangle,$$ we can write  $v=O_x(v)+\alpha X(x)$ for some $\alpha\in\mathbb{R}$. 
   Therefore, $$ DX_t(x)O_x(v)= O_{X_t(x)}DX_t(x)O_x(v) + \beta X(X_t(x)),$$
for some $\beta \in \R$. 
Thus, we obtain
\begin{eqnarray*}
DX_t(x)v &=& DX_t(x)O_x(v)+ DX_t(x)(\alpha X(x)) \\
&=& O_{X_t(x)}DX_t(x)O_x(v) + \beta X(X_t(x)) + \alpha DX_t(x)X(x) \\
&=& O_{X_t(x)}DX_t(x)O_x(v) + (\alpha + \beta)X(X_t(x)). 
\end{eqnarray*}
On the other hand,
$$DX_t(x)v = O_{X_t(x)}DX_t(x)v + DX_t(x)v - O_{X_t(x)}DX_t(x)v.$$
So, since $DX_t(x)v - O_{X_t(x)}DX_t(x)v\in\langle X(x)\rangle$, 
$$P_t(x)O_x(v)=O_{X_t(x)}DX_t(x)O_x(v) = O_{X_t(x)}DX_t(x)v\in\SE_{X_t(x)}.$$
This shows the $P_t$-invariance of $\SE$.
In addition,  
$$(\alpha + \beta)X(X_t(x)) = DX_t(x)v - O_{X_t(x)}DX_t(x)v.$$
Thus, we have
\begin{eqnarray*}
 \|DX_t(x)v \|&=&\| O_{X_t(x)}DX_t(x)O_x(v) + DX_t(x)v -O_{X_t(x)}DX_t(x)v\| \\
 &=& \|P_t(x)O_x(v)+(Id-O_{X_t(x)})DX_t(x)v \| \\
 &\geq& \|P_t(x)O_x(v)\| -\| Id-O_{X_t(x)}\|\cdot\| DX_t(x)v \| \\
 &\geq& \|P_t(x)O_x(v)\| - \| DX_t(x)v \|, 
\end{eqnarray*}
where the last inequality comes from the fact that $Id-O_{X_t(x)}$ is a
projection and  $$\|Id-O_{X_t(x)}\|\leq 1.$$
We then conclude that $$\|P_t(x)O_x(v)\| \leq 2 \| DX_t(x)v \|.$$

By taking the supreme over all unitary vectors in $\SE_x$, we have 
\begin{equation} \label{EE}
\|P_t(x)|_{\SE}\| \leq 2\| DX_t(x)|_{E}\|    
\end{equation}

By an analogous reasoning, we obtain the $P_t$-invariance of $\SF$ and 
\begin{equation} \label{FF}
 \|P_{-t}(x)|_{\SF}\| \leq 2\| DX_{-t}(x)|_{F}\|.   
\end{equation}
Now, by combining the domination property of the splitting $T_{\Lambda}M=E \oplus F$ with the relations (\ref{EE}) and (\ref{FF}), we have 
\begin{eqnarray*}
    \|P_{t}(x)|_{\SE}\|\cdot\|P_{-t}(x)|_{\SF}\| &\leq & 4 \| DX_{t}(x)|_{E}\|\cdot\| DX_{-t}(x)|_{F}\| \leq 4C e^{-\lambda t},
\end{eqnarray*} 
for every $t>0$. This proves the result.
\end{proof}
\begin{proof}[Proof of Theorem \ref{growth}]
Let $X$ be a $C^1$ vector field. If $h(X)=0$, then the result trivially holds. Therefore, let us suppose $h(X)>0$. Let $\mu$ be an ergodic measure supported on $\Lambda$ with positive topological entropy. This implies, in particular, that $\mu$ is regular. So, by Theorem \ref{erghyp}, $\mu$ is a hyperbolic measure. 

For every Lyapunov regular point $x$ for $\mu$,  Let $T_xM=E^s(x)\oplus\langle X(x)\rangle\oplus E^u(x)$ its Oseledets splitting, where 
$$E^s(x)=\bigoplus_{\chi_i(x)<0} H_{i}(x) \text{ and } E^u(x)=\bigoplus_{\chi_i(x)>0} H_{i}(x).$$
By Lemma \ref{E-E^c+}, we obtain $E(x)=E^s(x)$ and $E^c(x)=\langle X(x) \rangle \oplus E^u(x)$, so that the subbundle $\langle X(x) \rangle \oplus E^u(x)$ is dominated by $E^s(x)$.

Let us now proceed to conclude the proof. First, lift the hyperbolic splitting of $\mu$ to a splitting, $\SE\oplus \SE^c$ of $\SN$, by using the projection map $O$. By Lemma \ref{DDom} the splitting $\SE \oplus \SE^c$ is hyperbolic. In particular, it is dominated for the linear Poincaré flow $P_t$. Thus, we can apply Lemma \ref{PerEnt} to $\mu$ and obtain:
\begin{equation*}
    \limsup_{t \to \infty}\frac{1}{t}\log \# P_{t}(X) \geq h_{\mu}(X):=h_{\mu}(X_1).
\end{equation*}
Finally, by the Variational Principle, we conclude that  
\begin{equation*}
    \limsup_{t \to \infty}\frac{1}{t}\log \# P_{t}(X) \geq h(X):=h(X_1),
\end{equation*}
and the proof is complete.
\end{proof}

\subsection{Proof of Theorem \ref{StarinterASH}}
Next, we provide a proof for Theorem \ref{StarinterASH}.  We will achieve this by applying some of the results previously obtained.  Let $X\in \mathfrak{X}^1(M)$ and $\sigma\in Sing(X)$ be a hyperbolic singularity, and assume that the Lyapunov exponents of $DX_t(\sigma)$ are $\lambda_1\leq \cdots\leq \lambda_s<0<\lambda_{s+1}\leq \cdots\leq \lambda_d$. Recall from \cite{shiganwen} that the \textit{saddle value $sv(\sigma)$ of $\sigma$} is defined by $sv(\sigma)=\lambda_s+\lambda_{s+1}$. In this case, we say that a hyperbolic singularity $\sigma$ is 
\begin{itemize}
    \item \textit{Lorenz-like}: if $\lambda_s+\lambda_{s+1}>0$.
    \item \textit{Rovella-like}: if $\lambda_s+\lambda_{s+1}<0$.
    \item \textit{Resonant}: if $\lambda_s+\lambda_{s+1}=0$. 
\end{itemize}

We say that an attractor $\Lambda$ for a $C^1$ vector field $X$ is \textit{singular} if $Sing_{\Lambda}(X)=Sing(X)\cap\Lambda\neq\emptyset$. It is easy to check that if a singular attractor $\Lambda$ is ASH, then its singularities are of one of the three types mentioned above.  
In particular, the following result property is obtained. 

\begin{lemma} \label{xinW^s}
    Let $\Lambda$ be a non-trivial ASH attractor and let $\sigma\in Sing_{\Lambda}(X)$. Then, $(W^s(\sigma)\setminus \{ \sigma \}) \cap \Lambda \neq \emptyset$ and $(W^u(\sigma)\setminus \{ \sigma \}) \cap \Lambda \neq \emptyset$.
\end{lemma}

\begin{proof}
    Since $\sigma\in\Lambda$, we have that it is hyperbolic of saddle type. In this way, we can find $\varepsilon>0$ such that, for every $z \in M$, if $d(X_t(z),\sigma),\sigma)\leq \varepsilon$ for all $t\geq 0$, then $z\in W^s(\sigma)$. Since $\Lambda$ is non-trivial and transitive, there is a regular point $y \in \Lambda$ whose orbit is dense in $\Lambda$.
    
    Notice that given a neighborhood $U$ of $\sigma$, the closer a point $z \in U$ is to $\sigma$, the longer is the flight time of its orbit in $U$. Thus, it is possible to find a sequence of points $y_n\in\mathcal{O}(y)$ and a arbitrarily large sequence of positive numbers $t_n$, $n\geq 1$, such that $y_n \rightarrow \sigma$ and $$d(X_{t-t_n}(y_n),\sigma) \leq \varepsilon, \forall t\in [0,t_n]\text{ and }X_{-t_n}(y_n) \in B(\sigma, \varepsilon) \setminus B\left(\sigma,\frac{\varepsilon}{2}\right),\quad \forall n\geq 1.$$
    
    Therefore, for every accumulation point $x$ of $\{X_{-t_n}(y_n)\}_{n\geq 1}$ one has by continuity of the flow that $\frac{\varepsilon}{2} \leq d(x,\sigma) \leq \varepsilon$ and $d(X_t(x),\sigma)\leq \varepsilon$ for all $t\geq 0$. This shows that $x\in (W^s(\sigma)\setminus \{ \sigma \}) \cap \Lambda.$ Similar arguments prove that $(W^u(\sigma)\setminus \{ \sigma \}) \cap \Lambda \neq \emptyset$. 
\end{proof}

Now, let $V$ be a finite-dimensional vector space. We denote by $\wedge^2 V$ the second exterior power of $V$, defined as follows: If $v_1, \ldots, v_n$ be a basis of $V$, then $\wedge^2 V$ is the vector space spanned by the set $\left\{v_i \wedge v_j\right\}_{i \neq j}$. In this way, any linear transformation $A: V \rightarrow W$ induces a transformation $\wedge^2 A: \wedge^2 V \rightarrow \wedge^2 W$. Moreover, the element $v_i \wedge v_j$ can be viewed as the 2-plane generated by $v_i$ and $v_j$ if $i \neq j$. See for instance \cite{ASS} for more details.

\begin{definition} The \textit{sectional Lyapunov exponents of $x$} along $E^c$ are the limits
$$\lim_{t \rightarrow \infty }\frac{1}{t}\log \Vert\wedge^2 DX_t(x)\cdot \Tilde{v}\Vert,$$
 whenever they exist, where $\tilde{v} \in \wedge^2 E^c_x \setminus \{0 \}$   
\end{definition}

It turns out that if $\mu$ is an invariant probability measure, $Y$ is the subset given by Oseledec's theorem \ref{Oseledecflow} and $\left\{\chi_i(x)\right\}_{i=1}^{k(x)}$ are the Lyapunov exponents of a point $x \in Y$, then its sectional Lyapunov exponents are given by $\displaystyle\left\{\chi_i(x)+\chi_j(x)\right\}_{1 \leq i<j \leq k(x)}$. Moreover,  if $L_x$ is a 2-plane, then it can be seen as $\widetilde{v} \in \Lambda^2E^c_x \setminus \{0\}$ of norm one. In this way, the asymptotic expansion given in definition of ASH set can be rewritten as follows: There exists $C>0$ such that
$$
\limsup_{t \rightarrow \infty}\frac{1}{t}\log \left\|\wedge^2 DX_t(x) \cdot \tilde{v}\right\|\geq C>0,\quad \forall x\in\Lambda\setminus W^s(Sing(X)).
$$

Next, we present a useful result for proving Theorem \ref{StarinterASH}. But before, let us set some notation. Recall that if $\Lambda$ is ASH, then the ASH splitting of $\Lambda$ is denoted by $T\Lambda=E\oplus E^c$. On the other hand, any singularity contained on $\Lambda$ is hyperbolic and it has a hyperbolic splitting. We shall denote the hyperbolic splitting given by $T_{\sigma}M=\overline{E}^s_{\sigma}\oplus \overline{E}^u_{\sigma}$.

\begin{lemma}\label{lorenz}
    Let $\Lambda$ an asymptotically sectional-hyperbolic attractor for a $C^1$ vector field $X$. Suppose:
    \begin{enumerate}
        \item All the singularities contained in it are Lorenz-like.
        \item For every $\sigma\in Sing(X)$, one has $\dim(E(\sigma))=\dim(\overline{E}^s_{\sigma})+1$.
        \end{enumerate}
 Then, $\Lambda$ is sectional-hyperbolic.  
\end{lemma}

Before to present the proof of above lemma, we state the following theorem, proved by the first author in \cite{Arb}. 

\begin{theorem} \label{Arbieto}
Let $\{X_t\}_{t \in \R}$ be a flow with a dominated splitting $T_{\Lambda}M=E\oplus E^c$ over an attracting set $\Lambda$ whose singularities contained in it are hyperbolic. The flow $\{X_t\}_{t \in \R}$ is a sectional-hyperbolic flow if and only if the Lyapunov exponents in the $E$ direction are negative and the sectional Lyapunov exponents in the $E^c$ direction are positive on a set of total probability. If the manifold has no boundary, the flow has no singularities and it is an hyperbolic flow.
\end{theorem}

\begin{proof}[proof of Lemma \ref{lorenz}]
Suppose we are under the hypothesis of the Lemma. Let $\mu$ be a regular ergodic measure for the flow of $X$, and let $B\subset M$ be the set of regular points given by Oseledets theorem. Then, it is hyperbolic by Lemma \ref{erghyp} and, in particular, by Lemma \ref{E-E^c+} we have
$$E(x)=\bigoplus_{\chi(x)<0} H_{i}(x) \text{ and } E^c(x)=\bigoplus_{\chi(x)\geq0} H_{i}(x).$$
Then, the Lyapunov exponents in $E(x)$ are negative for $x\in B$. Now, for 
$x\in B \setminus W^s(Sing(X))$, let $u,v \in E^c(x)$ and let consider $L=\langle u, v\rangle$ the plane spanned by $u$ and $v$. Denote by $A(u,v)$ the area of the parallelogram defined by the vectors $u$ and $v$. We have that
\begin{align}\label{Area}
    A(DX_t(x)u,DX_t(x)v)= & |\det DX_t(x)\mid_L| A(u,v) \\
 \nonumber    = & \sin \theta_t\cdot\Vert DX_t(x)u\Vert\cdot\Vert DX_t(x)v\Vert,
\end{align}
where $\theta_t(x)$ represents the angle between $DX_t(x)u$ and $DX_t(x)v$. Then, the asymptotic area expansion on $L$ implies 
\begin{equation*}
\limsup_{t \rightarrow \infty} \frac{1}{t} \log \Vert \det DX_{t}(x)|_{L}\Vert \geq C>0.
\end{equation*}

On the other hand, by definition of the Lyapunov exponents and (\ref{Area}),
$$\limsup_{t \rightarrow \infty} \frac{1}{t} \log \Vert\det DX_{t}(x)|_{L}\Vert= \chi(x,u)+\chi(x,v),$$
so that $\chi(x,u)+\chi(x,v) \geq C >0$. Thus, the Lyapunov exponent sectional are positive. Since $\mu(W^s(Sing(X))\setminus Sing(X))=0$ for any $\mu$ measure invariant, by the recurrence Poincare theorem we assume that $B\cap (W^s(Sing(X))\setminus Sing(X))=\emptyset$. So, we need to verify that the property of having positive sectional lyapunov exponents is satisfied in the singularities. For this, recall that since $E$ is contracting one has $E(\sigma)\subset \overline{E}^s_{\sigma}$. On the other hand, by assumption we have
$$\dim (E(\sigma))+1= \dim(\overline{E}^s_{\sigma}),$$
which implies that $E^c(\sigma)$ contains one contracting direction.
In this way, the Lyapunov exponents in $E(\sigma)$ are $\lambda_1\leq \cdots\leq \lambda_{s-1}<0$, and the Lyapunov exponent in $E^c(\sigma)$ are $\lambda_s<0<\lambda_{s+1}\leq \cdots\leq \lambda_d$. Therefore, the sectional Lyapunov exponent are $\{\lambda_i+\lambda_j \}_{s-l\leq i<j\leq d}$ which are positive because $\sigma$ is Lorenz-like.  Then, by Theorem \ref{Arbieto}, we have that $\Lambda$ is sectional-hyperbolic.
\end{proof}

The last ingredient in the proof of Theorem \ref{StarinterASH} is the connecting lemma for chains, given in \cite{shiganwen}. 

\begin{lemma}[Lemma 2.4 in \cite{shiganwen}]\label{5.2} Let $X \in \mathcal{X}^*\left(M\right)$. For any $C^1$ neighborhood $\mathcal{U}$ of $X$ and $x, y \in M^d$, if $y$ is chain attainable from $x$, then there exists $Y \in \mathcal{U}$ and $t>0$ such that $\phi_t^Y(x)=y$. Moreover, for every $k \geq 1$, let $\left\{x_{i, k}, t_{i, k}\right\}_{i=0}^{n_k}$ be $a(1 / k, T)$-chain from $x$ to $y$ and denote by
$$
\Lambda_k=\bigcup_{i=0}^{n_k-1} \phi_{\left[0, t_{i, k}\right]}\left(x_{i, k}\right) .
$$
 Let $\Lambda$ be the upper Hausdorff limit of $\Lambda_k$, i.e., $\Lambda$ consists of points $z$ such that there exist $z_k \in \Lambda_k$ and $\lim _{k \rightarrow \infty} z_k=z$. Then for any neighborhood $U$ of $\Lambda$, there exists $Y \in U$ with $Y=X$ on $M \backslash U$ and $t>0$ such that $\phi_t^Y(x)=y$.
\end{lemma}

Next, we are ready to prove Theorem  \ref{StarinterASH}.

\begin{proof}[Proof of Theorem \ref{StarinterASH}]
Let $X$ be a star $C^2$-vector field, and let $\Lambda$ be an ASH attractor for $X$. If $\Lambda$ is non-singular, the Hyperbolic lemma  shows that it is, in fact, hyperbolic of saddle type. So, $\Lambda$ is sectional-hyperbolic.

Now, suppose that the ASH attractor $\Lambda$ is singular. Then, since $\Lambda$ is star, no singularity of $\Lambda$ can be resonant according to \cite{shiganwen}*{Corollary 4.3}. Therefore, in light of Lemma \ref{lorenz}, to conclude the proof we need to show the following two assertions:
\begin{enumerate}
    \item Every singularity in $\Lambda$ is Lorenz-like
    \item $\dim (E(\sigma))+1= \dim(\overline{E}^s_{\sigma}),$ for every $\sigma\in \Lambda$
\end{enumerate}

To see that every singularity is Lorenz-like we recall that since $X$ is star, there is a neighborhood  $\mathcal{U}\subset\mathcal{X}^1(M)$ of $X$ such that every singularity or periodic orbit of $Y\in \SU$ is hyperbolic. Assume that $\Lambda$ contains a singularity $\sigma$ that is not Lorenz-like. Then, as noted in the above paragraph, $\sigma$ must be Rovella-like.

Now, by Lemma \ref{posentr} we have $h(X\vert_{\Lambda})>0$, which implies, by Theorem \ref{growth}, that there exists a periodic point $p\in \Lambda$. In addition, the Hyperbolic lemma shows that $O(p)$ is hyperbolic of saddle type. By shrinking $\SU$ if it is possible, we can assume that the continuation $\sigma_Y$ of $\sigma$, for any $Y\in \SU$,  is also Rovella-like, and the continuation $p_Y$ of $p$ is also hyperbolic of saddle type. In this way, by Lemma \ref{xinW^s}, for every $\varepsilon>0$ we can find  points $$x_s^\sigma \in W^s(\sigma), x_u^\sigma \in W^u(\sigma), x_s^p \in W^s(p), x_u^p \in W^u(p)$$ and regular orbits $\gamma_1$ and $\gamma_2$ such that $$d\left(\gamma_1,\left\{x_u^\sigma, x_s^p\right\}\right),  d\left(\gamma_2,\left\{x_u^p, x_s^\sigma\right\}\right)<\varepsilon.$$ So, by Lemma \ref{5.2}, there is a vector field $Y \in \mathcal{U}$ admitting  a singular cycle $\Gamma_Y$ associated to the continuation $\sigma_Y$  and  $p_Y$. More precisely, there are regular orbits $\gamma_1^Y \in W^u\left(\sigma_Y\right) \cap W^s\left(p_Y\right)$ and $\gamma_2^Y \in W^u\left(p_Y\right) \cap W^s\left(\sigma_Y\right)$ such that $$\Gamma_Y=\left\{\sigma_Y\right\} \cup\left\{p_Y\right\} \cup \gamma_1^Y \cup \gamma_2^Y.$$ 

Next, we claim that there is a one-parameter family of vector fields $\displaystyle\left\{Y_r\right\}_{r_0 \leq r \leq r_1}$ contained in  $\mathcal{U}$, such that $\Gamma_{r_0}$ is a periodic sink for $Y_{r_0}$ and $\Gamma_{r_1}$ is a periodic saddle for $Y_{r_1}$. Indeed, the construction of such vector fields can be found in \cite{shiganwen}*{Lemma 4.5}.
 Therefore, by the continuity of Lyapunov exponents of $\Gamma_r$ along $\{Y_r\}$, there is $r^{\prime} \in\left(r_0, r_1\right)$ such that  $\Gamma_{r^{\prime}}$ is a non-hyperbolic periodic orbit, contradicting the star property of $X$. Thus, every singularity is Lorenz-like. This proves (1).

Let us now prove $\dim (E^s(\sigma))+1= \dim(\overline{E}^s_{\sigma}),$ for every $\sigma\in \Lambda$. To begin with, we recall that, by the Hyperbolic lemma, if $K\subset M$ is compact, invariant and non-singular, then it is a hyperbolic set. In this case, we have that  $\dim(E^s(x))=\dim(\overline{E}^s_x)$, for every $x\in K$, where $\overline{E}^s_x$ is the hyperbolic contracting bundle of $x$. Indeed, we first have $E^s(x)\subset \overline{E}^s_x$, and if this inclusion is proper, $E^c(x)$ must contain a two-dimensional subset that is not asymptotically sectional-expanding. In particular, this holds for any periodic orbit of $\Lambda$. Let $U$ be an isolating neighborhood of $\Lambda$. Since $X$ is star, again by following the steps of the proof of \cite{shiganwen}*{Lemma 4.5} one can show that if $Y$ is $C^1$-close to $X$ and $p\in U$ is a periodic point for $Y$, then it is hyperbolic of saddle type with $\dim(\overline{E}^s_Y(p))=\dim(E^s)$.

Next, since every singularity in $Sing_{\Lambda}(X)$ is Lorenz-like (by (1)), the continuation of any of these singularities, for vector fields $Y\in\mathcal{U}$, is also Lorenz-like with the same stable index. Next, we proceed as in the proof of step $(1)$ to a find a $C^1$-close vector field $Y$ a singular cycle $\Gamma_Y$ and a sequence vector fields $Y_n$ and periodic orbits $p_n\in U$ for $Y_n$ such that 
$Y_n\to Y$ and the orbits of $p_n$ converges to $\Gamma_Y$ in the Hausdorff topology. Finally, by \cite{shiganwen}*{Lemma 4.4}, $$\dim(\overline{E}^s_{\sigma})=\dim(\overline{E}^s_{\sigma_Y})=\dim(\overline{E}_{Y_n}^s(p_n))+1=\dim(E^s)+1.$$  
Therefore, the property (2) holds. This concludes the proof.
\end{proof}

\section{Three-dimensional ASH attractors}\label{3d}

In this section, we prove Theorem \ref{Entrexp} and Theorem  \ref{HomClass3}.

\subsection{Entropy Expansiveness} \label{SUBEntropyExpansiveness}
The proof of Theorem \ref{Entrexp} is essentially based on three results. The first two of them are elementary facts from flow theory, and we present their proofs here for the sake of completeness.  

\begin{lemma}\label{timeone}
Let $\Phi=\{X_t\}_{t\in\mathbb{R}}$ be a continuous flow on a  compact metric space $M$. For every $\alpha>0$, there exists $\beta>0$ such that if $y\in B^{\infty}_{\beta}(x, X_1)$,  then $$d(X_t(x),X_t(y))\leq\alpha$$
\end{lemma}
\begin{proof}
Fix $\alpha>0$. Since $M$  is compact and $\Phi$ is a continuous flow, we can find $\beta>0$ such that if $d(x,y)\leq\beta$, then $d(X_t(x),X_t(y))\leq \alpha$, for every $t\in [0,1]$. Take $y\in B^{\infty}_{\beta}(x,X_1)$ and fix $t\in \mathbb{R}$. One can write $t=n_t+r_t$, with $n_t\in\mathbb{Z}$ and $0\leq r_t<1$. By hypothesis, we have $d(X_{n_t}(x),X_{n_t}(y))\leq\beta$ and so  $$d(X_{t}(x),X_{t}(y))=d(X_{n_t+r_t}(x),X_{n_t+r_t}(y))=d(X_{r_t}(X_{n_t}(x)),X_{r_t}(X_{n_t}(y)))\leq\alpha.$$
This concludes the proof.
\end{proof}

Notice that if one denotes $$B^{\infty}_{\alpha}(x,X)=\{y\in M;d(X_t(x),X_t(y))\leq\alpha, \forall t\in \mathbb{R}\},$$
then the previous lemma says  that  $B^{\infty}_{\beta}(x,X_1)\subset B^{\infty}_{\alpha}(x,X)$, if $\beta>0$ is small enough. 

\begin{lemma}\label{orbitentropy}
Let $\Phi=\lbrace X_t\rbrace_{t\in\mathbb{R}}$ be a continuous flow on a compact metric space $M$. Then, $h(X_{[-\varepsilon,\varepsilon]}(x))=0$ for every   $\varepsilon >0$ and  $x \in M$. 
\end{lemma}

\begin{proof}
First, we claim that for every $\eta>0$ there is $\varepsilon_0>0$ such that if $y\in X_{[-\varepsilon_0,\varepsilon_0]}(x)$, then $d(X_t(x),X_t(y))\leq \eta$, for any $t\in \mathbb{R}$. Indeed, otherwise there is $\eta>0$ and sequences $x_n\in M$, $t_n\in\mathbb{R}$ and $s_n\in\mathbb{R}$ with $s_n \to 0$ such that
\begin{equation}\label{sepa}
    d(X_{t_n}(x_n),X_{t_n+s_n}(x_n))>\eta,\quad\forall n\geq 1.
\end{equation}
By compactness of $M$ and the continuity of the flow, there is $s>0$ such that $d(X_t(x),x)<\eta$ for any $x\in M$ and $\vert t\vert<s$. So, by \eqref{sepa} we obtain 
\begin{displaymath}
\eta>d(X_{s_n}(X_{t_n}(x_n)),X_{t_n}(x_n)))=d(X_{s_n+t_n}(x_n),X_{t_n}(x_n))\geq \eta
\end{displaymath}
for $n$ large enough, which leads us to a contradiction.

Now, we fix $\eta>0$. By the above claim, there is $\varepsilon_0>0$ such that 
\begin{displaymath}
d(X_t(x),X_t(y))\leq \eta,\quad \forall t\in\mathbb{R}, 
\end{displaymath}
for every $x\in M$ and $y\in X_{[-\varepsilon_0,\varepsilon_0]}(x)$. So, if $\varepsilon_0>\varepsilon$, we have $S(t,\eta)=1$, where $S(t,\eta)$ denotes the minimum cardinality of a $(t,\eta)$-spanning set, whereas if $\varepsilon_0<\varepsilon$, there is $N\in\mathbb{N}$ such that $N\varepsilon_0>\varepsilon$, so that $S(t,\eta)\leq N$. Then, we have the desired result by the definition of topological entropy. 
\end{proof}

In light of the above lemmas, the proof of Theorem \ref{Entrexp} is reduced to obtain the following result:

\begin{theorem}\label{k-exp}
Every asymptotically sectional-hyperbolic attractor $\Lambda$ associated to $C^1$ vector field $X$ on a three-dimensional manifold $M$ is \textit{kinematic expansive}, i.e, for every $\varepsilon>0$, there is $\delta>0$ such that if $x,y\in \Lambda$ satisfy 
$$d(X_t(x),X_t(y))\leq \delta\quad \forall t\in \mathbb{R},$$ then $y\in X_{[-\varepsilon,\varepsilon]}(x)$.
\end{theorem}

\begin{remark}
Note that kinematic expansiveness is a weaker form of expansiveness since it does not care about reparametrizations. For a more detailed discussion about the properties of kinematic expansive flows, we refer the reader to the work \cite{Ar1}.
\end{remark}
We explain why these results imply Theorem \ref{Entrexp}.

\begin{proof}[Proof of  Theorem \ref{Entrexp}.]
Fix $\eta>0$. By Theorem \ref{k-exp} there is $\varepsilon>0$ such that $B_{\varepsilon}^{\infty}(x)\subset X_{[-\eta,\eta]}(x)$ for any $x\in M$. Let $\beta>0$ be as in Lemma \ref{timeone} with respect to $\varepsilon$. Then, for every $x\in M$ one has $$B_{\beta}^{\infty}(x,X_1)\subset B_{\varepsilon}^{\infty}(x,X)\subset X_{[-\eta,\eta]}(x),$$ and the proof follows from Lemma \ref{orbitentropy}. 
\end{proof}

From now on, we will devote ourselves to obtaining a proof of Theorem \ref{k-exp}. Many parts of the arguments presented here resemble that of the proof of Theorem 2.5 in \cite{RV}. First, we state some known results that will be used in the proof. 

Let $\sigma$ be a singularity of $\Lambda$. We say that $\sigma$ is 
\begin{itemize}
    \item \textit{attached} to $\Lambda$ if it is accumulated by regular orbits in $\Lambda$, 
    \item \textit{real index two} if it has three real eigenvalues satisfying  $\lambda_{ss}<\lambda_s<0<\lambda_u$.
\end{itemize}
For a real index-two singularity $\sigma$, we say that it is \textit{Lorenz-like} if its eigenvalues satisfy the relation $\lambda_u+\lambda_s>0$, called \textit{central expanding condition}. In \cite{B}, it was shown that the singularities contained in a connected singular-hyperbolic set must be Lorenz-like. This is not true for ASH sets in general (see \cite{SMV} for instance). For ASH sets, we state the following result, whose proof is analogous to that of Theorem A in \cite{Mo3}:
\begin{theorem}\label{T1}
Let $\Lambda$ be a nontrivial asymptotically sectional-hyperbolic set of $X$ and assume that $\Lambda$ is not hyperbolic. Then, $\Lambda$ has at least one attached singularity. In addition, if $\Lambda$ is transitive, the following holds for $X$: Each singularity $\sigma$ of $\Lambda$ is real of index two and satisfies 
\begin{equation}\label{C}
\Lambda\cap W^{ss}(\sigma)=\lbrace\sigma\rbrace,
\end{equation}
where $W^{ss}(\sigma)$ denotes the strong stable manifold associated with $\lambda_{ss}$.
\end{theorem}

Now, let $E^s\oplus E^c$ be the partially hyperbolic splitting associated with the asymptotically sectional-hyperbolic attractor $\Lambda$, and consider a continuous extension $\widetilde{E}^s\oplus \widetilde{E}^c$ to the trapping region $U_0$. According to Proposition 3.2 in \cite{AM}, the subbundle $\widetilde{E}^s$ can be chosen $DX_t$-invariant for positive $t$. Nevertheless, the subbundle $\widetilde{E}^c$ is not invariant in general, but we can consider a cone field $C_a^c$ of size $a>0$ around $\widetilde{E}^c$ on $U_0$ defined by 
\begin{displaymath}
C_a^c(x):=\lbrace v=v_s+v_c : v_s\in \widetilde{E}^s, v_c\in \widetilde{E}^c\text{ and } \Vert v_s\Vert\leq a\Vert v_c\Vert\rbrace,\quad \forall x\in U_0,
\end{displaymath}
which is invariant for $t>0$ large enough, i.e., there is $T_0>0$ such that 
\begin{displaymath}
DX_t(x)C_a^c(x)\subset C_a^c(X_t(x)),\quad\forall t\geq T_0,\forall x\in U_0. 
\end{displaymath}
\begin{remark}
By possibly shrinking the neighborhood $U_0$, the number $a>0$ can be taken arbitrarily small.
\end{remark}
To simplify the notation, we write $E^s$ and $E^c$ for $\widetilde{E}^s$ and $\widetilde{E}^c$ respectively in what follows.

Now, we consider a special kind of neighborhood of the singularities contained in an ASH attractor. Let $\Lambda$ be an ASH attractor on a three-dimensional manifold $M$. Let $x\in\Lambda$ and $\Sigma'$ be a cross-section to $X$ containing $x$ in its interior. Define $W^s(x,\Sigma')$ as the connected component of $W^s(x)\cap\Sigma'$. This gives us a foliation $\mathcal{F}_{\Sigma'}$ of $\Sigma'$. We can construct a smaller cross section $\Sigma$, which is the image of a diffeomorphism $h:[-1,1]\times[-1,1]\to \Sigma'$, that sends vertical lines inside $\mathcal{F}_{\Sigma'}$ in a such way that $x$ belongs to the interior of $h([-1,1]\times[-1,1])$. In this case, the \textit{s-boundary} $\partial^s\Sigma$ and \textit{cu-boundary} $\partial^{cu}\Sigma$ of $\Sigma$ are defined by 
\begin{displaymath}
\partial^s\Sigma=h(\lbrace -1,1\rbrace\times[-1,1])\quad\text{ and }\quad\partial^{cu}\Sigma=h([-1,1]\times\lbrace-1,1\rbrace)
\end{displaymath}
respectively. We say that $\Sigma$ is $\eta$-\textit{adapted} if 
\begin{displaymath}
d(\Lambda\cap\Sigma,\partial^{cu}\Sigma)>\eta. 
\end{displaymath}

A consequence of the Hyperbolic lemma (see \cite{SMV2}) is the following: 
\begin{prop}
Let $\Lambda$ be an asymptotically sectional-hyperbolic attractor, and let $x\in\Lambda$ be a regular point. Then,
there exists a $\eta_0$-adapted cross-section $\Sigma$ at $x$ for some $\eta_0>0$. 
\end{prop}

\begin{remark}
From any  $x\in M$ and any cross-section containing $x$ in its interior, one can obtain an  $\eta_0$-adapted cross-section containing $x$ in its interior.
\end{remark}

Next, we recall the construction performed in \cite{SYY} of partitions for singular cross sections. These partitions give us a detailed picture of the flow dynamics inside small neighborhoods of the singularities of $\Lambda$.   According to that reference we can find  $\beta_1>0$ such that:

\begin{enumerate}[(a)]
    \item $B_{\beta_1}(\sigma)\cap B_{\beta_1}(\sigma')=\emptyset$, where $B_r(a)$ denotes the open ball centered in $a$ and radius $r>0$ and $\sigma,\sigma'\in Sing_{\Lambda}(X)$.
    \item The map $exp_{\sigma}$ is well defined on $\lbrace v\in TM_{\sigma} : \Vert v\Vert\leq\beta_1\rbrace$ for every $\sigma\in Sing_{\Lambda}(X)$. 
    \item There are $L_0, L_1>0$ such that 
    \begin{displaymath}
    L_0\leq\frac{\Vert X(x)\Vert}{d(x,\sigma)}\leq L_1,\quad\forall x\in \overline{B_{\beta_1}(\sigma)},\quad\forall \sigma\in Sing_{\Lambda}(X). 
    \end{displaymath}
    \item The flow in $B_{\beta_1}(\sigma)$ is a small $C^1$ perturbation of the linear flow. 
\end{enumerate}
For every $\sigma\in Sing_{\Lambda}(X)$, define the \textit{singular cross-section}
\begin{displaymath}
D_{\sigma}=exp_{\sigma}(\lbrace v=(v^s,v^u)\in TM_{\sigma} : \Vert v\Vert\leq\beta_1, \Vert v^s\Vert=\Vert v^u\Vert \rbrace)\subset M,
\end{displaymath}
and the following partition of $D_{\sigma}$
\begin{displaymath}
D_n=D_{\sigma}\cap(B_{e^{-n}}(\sigma)\setminus B_{e^{-(n+1)}}(\sigma)),\quad\forall n\geq n_0,
\end{displaymath}
where $n_0$ is large enough. 
As noticed by the authors in \cite{SYY}, the partition $\{D_n\}_{n\geq n_0}$ induces  a partition of the  cross sections $\Sigma_{\sigma}^{i,o,\pm}$ given in \cite{AP}. More precisely, assume that $\Sigma_{\sigma}^{i,o,\pm}\subset\partial B_{\beta_1}(\sigma)$. Consider 
\begin{displaymath}
D_n^o=\bigcup_{x\in D_n}X_{t_x^+}(x), \quad D_n^i=\bigcup_{x\in D_n}X_{-t_x^-}(x),\quad\forall n\geq n_0,
\end{displaymath}
where 
\begin{displaymath}
t_x^+=\inf\lbrace\tau>0 : X_{\tau}(x)\in \Sigma_{\sigma}^{o,\pm} \rbrace,
\end{displaymath}
and
\begin{displaymath}
t_x^-=\inf\lbrace\tau>0 : X_{-\tau}(x)\in \Sigma_{\sigma}^{i,\pm} \rbrace.
\end{displaymath}
Note that the sets $\lbrace D_n^i\cap\Sigma_{\sigma}^{i,\pm}\rbrace_{n\geq n_0}$, illustrated in the Figure 1, form a partition of $\Sigma_{\sigma}^{i,\pm}$ for which $X_{\tau(x)}(x)\in D_n^o\cap\Sigma_{\sigma}^{o,\pm}$ for every  $x\in D_n^i\cap\Sigma_{\sigma}^{i,\pm}$ and every $n\geq n_0$, where $\tau(\cdot)$ is the flight time to go from $\Sigma_{\sigma}^{i,\pm}$ to $\Sigma_{\sigma}^{o,\pm}$. Moreover, it is  shown that this flight time satisfies 
\begin{equation}\label{vuelo}
\tau(x)\approx \left(\frac{\lambda_u+1}{\lambda_u}\right)\cdot n,\quad \forall x\in D_n^i\cap\Sigma_{\sigma}^{i,\pm},\quad\forall n\geq n_0.
\end{equation}

\begin{figure}[ht]
\includegraphics[scale=0.4]{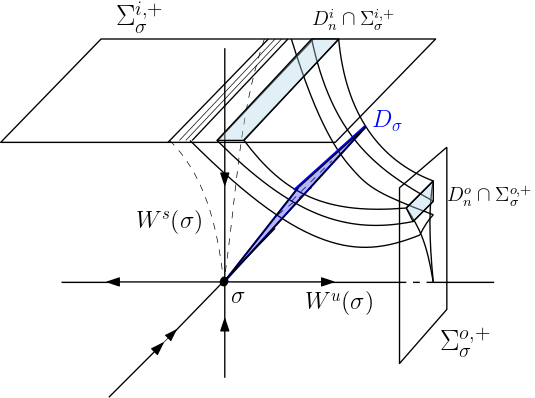}
\caption{The sets $D^{i,o}_n\cap\Sigma_{\sigma}^{i,o,\pm}$.}
\end{figure}

In this case, let $\widetilde{\Sigma_{\sigma}^{i,\pm}}=\left(\left(\bigcup_{n\geq n_0}D_n^i\cap\Sigma_{\sigma}^{i,\pm}\right)\right)\cup \ell_{\pm}$, where $\ell_{\pm}$ is contained in $ W^s_{loc}(\sigma)\cap\Sigma_{\sigma}^{i,\pm}$, and let consider
\begin{equation}\label{vecindad}
V_{\sigma}=\bigcup_{z\in\widetilde{\Sigma_{\sigma}^{i,\pm}}\setminus\ell_{\pm}}X_{(0,\tau(z))}(z)\cup (-e^{-n_0},e^{-n_0})\times(-e^{-n_0},e^{-n_0})\times(-1,1)\subset U. 
\end{equation}

Denote $\widetilde{\Sigma^{i,o,\pm}}=\bigcup_{\sigma\in Sing_{\Lambda}(X)}\widetilde{\Sigma_{\sigma}^{i,o,\pm}}$ and $V=\bigcup_{\sigma\in Sing_{\Lambda}(X)}V_{\sigma}$.

\begin{remark}\label{rmk}
We have the following remarks: 
\begin{itemize}
    \item We can assume without loss of generality that every cross section in $\widetilde{\Sigma^{i,o,\pm}}$ is $\eta_0$ adapted for some $\eta_0>0$. 
    \item The above construction can be made by taking $\widetilde{\beta}<\beta_1$. In this case, we denote
    \begin{displaymath}
    \widetilde{\Sigma_{\widetilde{\beta}}^{i,o,\pm}}=\bigcup_{\sigma\in Sing_{\Lambda}(X)}\widetilde{\Sigma_{\sigma,{\widetilde{\beta}}}^{i,o,\pm}}\quad\text{and}\quad V_{\widetilde{\beta}}=\bigcup_{\sigma\in Sing_{\Lambda}(X)}V_{\sigma,{\widetilde{\beta}}}
    \end{displaymath}
\end{itemize}
\end{remark}

Before continuing with the proof, we will take a little break to briefly outline the following lemmas to make our argument clearer. Our main goal is to obtain some sort of hyperbolicity for Poincar\'{e} maps from the ASH property. In \cite{AP}, this property is obtained from the uniform area expansion property of the sectional hyperbolicity. Nevertheless, we can only see area expansion on the center bundle for ASH attractors during hyperbolic times. 

In what follows, we prove that if two points satisfy the shadowing condition given in the definition of kinematic expansiveness, they share the same hyperbolic times. To do this, we need some previous lemmas to help us control the hyperbolic times of a pair of close points, depending on where they are located in the phase space. More precisely, Lemma \ref{W_0} controls the hyperbolic times of points away from singularities, Lemma \ref{W_1} controls the hyperbolic times of points before entering a neighborhood of a singularity and  Lemma \ref{lema31} helps us to control the hyperbolic times of points crossing a neighborhood of a singularity.

First, denote $\Lambda_+=\bigcap_{t\geq 0}X_{-t}(\Lambda\setminus V)$. Then, we have the following lemma.

\begin{lemma}[\textit{Lemma 3.2 in}   \cite{RV}] \label{W_0}
Given $\varepsilon_0>0$, there are positive numbers $\delta_1(\varepsilon_0), T_0$ and $K_2=K_2(\varepsilon_0)\approx 1$, and a neighborhood $W_0$ of $\Lambda_+$ such that for any $x,y\in W_0$ with $d(x,y)<\delta_1(\varepsilon_0)$, then 
\begin{equation}\label{estimateW}
    \frac{\vert\det DX_{t_1}(y)\vert_{E^c_y}\vert}{\vert\det DX_{t_1}(x)\vert_{E^c_x}\vert}\geq K_2,\quad 0<t_1\leq T_0,  
\end{equation}

where $t_1\leq T_0$ is a first hyperbolic time for $x$ and  $y$. 
\end{lemma}

Let us consider the following compact set $\Lambda''=\Lambda\setminus(V\cup W_0)$, where $W_0$ is as the above lemma.  
\begin{lemma}[\textit{Lemma 3.3 in}  \cite{RV}] \label{W_1}
Given $\varepsilon_1>0$, there are positive numbers $\delta_2(\varepsilon_1),T_1$ and $K_3=K_3(\varepsilon_1)\approx 1$, and a neighborhood $W_1$ of $\Lambda''$ such that for every $x,y\in W_1$, with $d(x,y)\leq \delta_2(\varepsilon_1)$, there is $0<s\leq T_1$ such that $X_s(x),X_s(y)\in V$ and  
\begin{equation}\label{estimateoutside}
    \frac{\vert\det DX_s(y)\vert_{E^c_y}\vert}{\vert\det DX_s(x)\vert_{E^c_x}\vert}\geq K_3.
\end{equation}
\end{lemma}

We need, in addition, the following result

\begin{lemma}[\textit{Lemma 3.4 in } \cite{RV}]\label{conj}
Let $\widetilde{U}=\overline{U\setminus V}$, where $U$ is the basin of attraction of $\Lambda$. There are $\beta'>0$ and $\varepsilon_2>0$ such that if $x\in \widetilde{U}$, and $y,z\in B_{\varepsilon_2}(x)$, $z\in\mathcal{O}(y)$, then $z=X_u(y)$, $\vert u\vert<2\beta'$. 
\end{lemma}

\begin{lemma}[\textit{Lemma 3.5 in } \cite{RV}]\label{rL}
There exists $\delta_3 >0$ with the following property: For $x\in \Sigma_{\widetilde{\beta}}^{i,o,\pm}$, $\sigma \in Sing_{\Lambda}(X)$ and $z\in M$ with $d(x,z)<\delta_3$ exist $l \in \mathbb{R}$, with $|l|<L$, such that $X_l(z)\in\Sigma_{\widetilde{\beta}}^{i,\pm}$, where $L= \frac{\lambda_u+1}{\lambda_u}$.
 \end{lemma}

Next lemma is contained in the proof of Theorem 2.5 in \cite{RV}. For the reader convenience, we state it here separately and provide a proof.

\begin{lemma} \label{lema31} 
Let $\widetilde{\varepsilon}_0>0$. There are a positive number $\delta_0$, independent of $\widetilde{\varepsilon}_0$, and a positive number $\widetilde{\beta}$ such that if $x,y\in\widetilde{\Sigma_{\widetilde{\beta}}^{i,\pm}}$ satisfy $$d(X_t(x),X_t(y))\leq\delta_0 $$ for every $t\in\mathbb{R}$,  then 
\begin{equation}\label{estimneighborhood}
    \frac{\vert \det DX_{\tau(y)}(y)\vert_{E^c_y}\vert}{\vert \det DX_{\tau(x)}(x)\vert_{E^c_x}\vert}\geq K_1C_0^{\tau(x)}, 
\end{equation}
where $C_0=C_0(\widetilde{\varepsilon}_0)$  and $K_1$ is a fixed positive constant. 
\end{lemma}

\begin{proof}
Let $\sigma$ be an attached singularity of $\Lambda$. First, recall that since $\sigma$ is hyperbolic, we can use the Grobman-Hartman Theorem to conjugate $X_t$ with its linear part in a small neighborhood of $\sigma$. More precisely, there are a neighborhood $U_{\sigma}$ of $\sigma$, a neighborhood $V_0$ of $0\in T_{\sigma}M$ and a homeomorphism $h:U_{\sigma}\to V_0$ that conjugates $X_t$ with its  linear flow $L_t$. Furthermore, this homeomorphism can be chosen to be H\"older continuous (see \cite{BV} for more details). Thus, there are  $C=C(\sigma)>0$ and $\alpha(\sigma)>0$ such that 
\begin{equation}\label{Holder}
    \Vert h(x)-h(y)\Vert\leq Cd(x,y)^{\alpha},\quad\forall x,y\in U_{\sigma}. 
\end{equation}

Since we have only finitely many singularities in $\Lambda$, we can shrink  $U_{\sigma}$, $\sigma\in Sing_{\Lambda}(X)$, if it is necessary to obtain uniform H\"older constants for every singularity, i.e., we can write \eqref{Holder} with
\begin{displaymath}
C=\max_{\sigma\in Sing_{\Lambda}(X)}\lbrace C(\sigma)\rbrace\text{ and }\alpha=\min_{\sigma\in Sing_{\Lambda}(X)}\lbrace \alpha(\sigma)\rbrace.
\end{displaymath}

Now, let us consider  
\begin{displaymath}
W^{ss}(0)=\lbrace (a,b,c): b=c=0\rbrace,\quad W^s(0)=\lbrace (a,b,c): b=0\rbrace
\end{displaymath}
and 
\begin{displaymath}
W^u(0)=\lbrace (a,b,c): a=c=0\rbrace.
\end{displaymath}
inside of the neighborhood $V_0$. We can assume that $V_{\sigma}\subset U_{\sigma}$, where $V_{\sigma}$ is as in \eqref{vecindad}. In particular, we have that  $\Sigma_{\sigma}^{i,o,\pm}\subset U_{\sigma}$ for every $\sigma\in\Lambda$. Since $h$ is a homeomorphism, it induces a partition of $h(\Sigma_{\sigma}^{i,\pm})\subset U_{0}.$ Indeed, we just need to consider the family $$\mathcal{F}=\{h(D_n^i\cap\Sigma_{\sigma}^{i,\pm})\}_{n\geq n_0}.$$ 
By the remarks in page 380 of \cite{SYY}, there  is $K'>0$ such that for every $x\in D_n^i\cap\Sigma_{\sigma}^{i,\pm}$, we have  
    $$d(x,W^s(\sigma))\leq K'e^{-(\lambda_u+1)n}.$$  
Then, since $h(W^s(\sigma))=W^s(0)$, the H\"older  estimates of $h$ gives us 
\begin{eqnarray*}
d(h(x),W^s(0))&\leq& Cd(x,W^s(\sigma))^{\alpha}\\
&\leq &CK'e^{-\alpha(1+\lambda_u)n} \\
&=& e^{\left(\frac{\ln (CK')}{n}-\alpha(1+\lambda_u)\right)n}. 
\end{eqnarray*}

On the other hand, let $0<\varepsilon_0<\alpha(1+\lambda_u)$ and $n_0\geq 1$ such that  $$\vert\ln (CK')/n\vert<\varepsilon_0,\forall n\geq n_0.$$  Then, 
\begin{displaymath}
d(h(x),W^s(0))\leq e^{\rho n},\quad \forall x\in D_n^i\cap\Sigma_{\sigma}^{i,\pm},
\end{displaymath}
where $\rho=\varepsilon_0-\alpha(1+\lambda_u)<0$. In fact, since $\mathcal{F}$ is a partition of $h(\Sigma_{\sigma}^{i,\pm})$ we have that 
\begin{equation}\label{part}
e^{\rho(n+1)}\leq d(h(x),W^s(0))\leq e^{\rho n},\quad \forall x\in D_n^i\cap\Sigma_{\sigma}^{i,\pm},\forall n\geq n_0. 
\end{equation}
Therefore, if $$h(\Sigma_{\sigma}^{o,\pm})=\Sigma_0^{\pm}=\lbrace p=(\pm1,b,c):\vert b\vert,\vert c\vert<1\rbrace\subset U_0,$$
we obtain  by \eqref{part} that the flight time $\tau(p)$, $p\in h(\Sigma_{\sigma}^{i,\pm})$, to go from $h(\Sigma_{\sigma}^{i,\pm})$ to $\Sigma_0^{\pm}$ satisfies 
\begin{equation}\label{timexit}
    -\frac{\rho}{\lambda_u}n\leq\tau(p)\leq -\frac{\rho}{\lambda_u}(n+1). 
\end{equation}

Now, since $h$ conjugates $X_t$ and $L_{\sigma}$, by applying Lemma \ref{conj} we  obtain that for every  $\beta>0$ there is $\varepsilon>0$ such that if $v\in V_0$ and $w\in h(\Sigma_{\sigma}^{i,\pm})$  satisfy  $\Vert v-w\Vert<\varepsilon$, then $L_u(v)\in h(\Sigma_{\sigma}^{i,\pm})$, with $u\in(-\beta,\beta)$.

Take a compact neighborhood $C'$ of $V_0$ inside $U_0$. By uniform continuity of $h$ on $h^{-1}(C')$, there is $\delta>0$ such that $\Vert h(z)-h(w)\Vert<\varepsilon$ if $z,w\in h^{-1}(C')$ satisfy $d(z,w)<\delta$. Besides, by uniform continuity of $DX_1(\cdot)$ and $E^c$ on $\overline{U}$ there exists  $0<\delta_0<\delta$ such that 
\begin{equation}
    \text{if }\quad d(x,y)<\widetilde{\delta}_0,\quad\text{then}\quad \frac{\vert \det DX_1(y)\vert_{E^c_y}\vert}{\vert \det DX_1(x)\vert_{E^c_x}\vert}\geq C_0\approx 1.
\end{equation}

Let us consider 
\begin{displaymath}
0<\beta<-\rho/\lambda_u , \quad 0<\varepsilon<1-e^{\rho+\lambda_u\beta},
\end{displaymath}
and $x,y\in\Lambda$ such that $$d(X_t(x),X_t(y))\leq\delta_0 \quad \forall t \in \mathbb{R}.$$
Suppose that there is $t\geq0$ such that $$X_t(x)\in D_n^i\cap\Sigma_{\sigma}^{i,\pm},$$ for some $\sigma\in Sing_\Lambda(X)$ and $n\geq n_0$. Then, there is $s\in\mathbb{R}$ such that $y'=X_{s+t}(y)\in \Sigma_{\sigma}^{i,\pm}$. Moreover, both points $y'$ and $X_t(x)$ belongs to the same connected component of $\Sigma_{\sigma}^{i,\pm}\setminus\ell_{\pm}.$

\textbf{Claim:} $y'\in \left( D_{n-1}^i\cap\Sigma_{\sigma}^{i,\pm}\right)\cup \left( D_n^i\cap\Sigma_{\sigma}^{i,\pm}\right)\cup \left(D_{n+1}^i\cap\Sigma_{\sigma}^{i,\pm}\right)$. 

Indeed, assume that $y'\in D_{n+k}^i\cap\Sigma_{\sigma}^{i,\pm}$ for some $k>1$. By the choice of $\delta$ there is $u\in(-\beta,\beta)$ such that 
\begin{displaymath}
h(y')=(a_0,b_0,c_0)=(e^{\lambda_{ss}u}a,e^{\lambda_uu}b,e^{\lambda_su}c)\in h(D_{n+k}^i\cap\Sigma_{\sigma}^{i,\pm}).
\end{displaymath}
Moreover, assume that $a,c>0$. Since $h(x)\in D_n^i\cap\Sigma_{\sigma}^{i,\pm}$, we have by \eqref{part} and \eqref{timexit} that
\begin{eqnarray*}
\varepsilon>\Vert  L_{\tau(h(x))}(h(x))-L_{\tau(h(x))}(h(y))\Vert&\geq& 1-e^{\lambda_u\tau(h(x))}a\\
&=&1-e^{-\lambda_uu}e^{\lambda_u\tau(h(x))}a_0\\
&\geq& 1-e^{\rho+\lambda_u\beta},
\end{eqnarray*}
which contradicts the choice of $\varepsilon$, and the claim is proved. 

So, by the above claim, we have that.
   
   \begin{equation}\label{L}
\vert\tau(x)-\tau(y)\vert\approx \frac{\lambda_u+1}{\lambda_u}=L.
\end{equation}
  
 Let $n_1\geq 0$ be the largest integer such that $X_{n_1}(y)\in V_{\sigma}$. Thus we have $\tau(y)=n_1+r_y$ with $0\leq r_y\leq 1$ and $\tau(x)=n_1+r_x$ with $|r_x|\leq L+1$. Therefore, by the chain rule, we have 
 
 $$\frac{\vert \det DX_{\tau(y)}(y)\vert_{E^c_y}\vert}{\vert \det DX_{\tau(x)}(x)\vert_{E^c_x}\vert}\geq K_1 C_0^{\tau(x)},$$
 
  for some $K_1>0$. 
 \end{proof}

Once we have the previous lemmas, we can quickly adapt the proof of lemma 3.6 in \cite{RV} to obtain the desired control of the hyperbolic times.

 \begin{lemma}\label{jac}
There exist positive numbers $\delta_4$, $T$, $c_*$ such that if $x\in\Lambda'$, and $y\in U$ satisfy
\begin{displaymath}
d(X_t(x),X_{t}(y))\leq\delta_4,\quad\forall t\in\mathbb{R},
\end{displaymath}
and given a $C$-hyperbolic time $t_x\geq T$, we have  
\begin{equation}\label{expansion}
\vert\det DX_{t_x}(y)\vert_{E^c_y}\vert\geq e^{c_*t_x}.
\end{equation}
\end{lemma}
 
 \begin{proof}
 Let  $\widetilde{\varepsilon}_0, \varepsilon_0,\varepsilon_1$ be given by the previous lemmas and let  $c^*,\alpha>0$  satisfying
 \begin{equation} \label{c*}
C-(\alpha+|\ln(C_0)|+|\ln K_2|+|\ln K_3|) > c^*>0,     
 \end{equation}
where $K_2=K_2(\varepsilon_0)$ and $K_3=K_3(\varepsilon_1)$. Since $C_0 \approx 1$, $K_2 \approx 1$, $K_3 \approx 1$, it is possible to choose such $c^*$.
 Let $\delta_0, \delta_1=\delta_1(\varepsilon_0), \delta_2=\delta_2(\varepsilon_1)$ and $\delta_3$ be given by Lemmas \ref{lema31}, \ref{W_0} and \ref{rL}, respectively. Let consider
 $$0< \delta_4 < \min\left\lbrace\delta_0,\delta_1,\delta_2,\delta_3,\delta'\right\rbrace,$$
 where $\delta '$ is the Lebesgue's number of the open cover $V,W_0,W_1$ of $\Lambda$ and fix  $T=\max\{T_0,T_1\}$.
 
 Let $ x \in \Lambda'$ and $y \in U$ be two points satisfying the condition given the statement of lemma. Next, we will take a sequence of points $x_n$ in the orbit of $x$ which is contained in the union of $$((W_0 \cup W_1)\setminus V)\cup \widetilde{\Sigma_{\widetilde{\beta}}^{i,\pm}}.$$ Let $s_{0} \geq 0$ (if it exists) be the first time satisfying $$x_0 =X_{s_0}(x) \in\widetilde{\Sigma_{\widetilde{\beta}}^{i,\pm}},$$ where $\widetilde{\beta}$ is given by lemma \ref{lema31}. Then, define $$x_1=X_{s_0+\tau(x_0)}(x)$$
and  note that $x_1\in (W_0 \cup W_1)\setminus V$. Now we split the definition of $x_2$ into two cases:
\begin{enumerate}
    \item If $x_1\in W_1\setminus V$, then  Lemma \ref{W_1} implies the existence of $0<s\leq T$ such that $X_s(x) \in V$, and therefore we can take $0<s_1<s$ such that $X_{s_1}(x_1) \in\widetilde{\Sigma_{\widetilde{\beta}}^{i,\pm}}$. In this case,  define $x_2=X_{s_1}(x_1)$.
    \item If $x_1\in W_0\setminus V$, then  we consider $0<t_1\leq T$ given by Lemma \ref{W_0} and define  
  $x_2=X_{t_1}(x_1)$, if $X_{t_1}(x_1) \notin V$ or $x_2=X_{r_1}(x_1)$, if $X_{t_1}(x_1)\in V$ and $0<r_1<t_1$ is such that $$X_{r_1}(x_1) \in\widetilde{\Sigma_{\widetilde{\beta}}^{i,\pm}}.$$
  
\end{enumerate}
  Proceeding inductively for $n\geq3$, we define  

$$x_n= \left\{ \begin{array}{lcl}
     X_{\tau(x_{n-1})}(x_{n-1}) &  if  & x_{n-1} \in \widetilde{\Sigma_{\widetilde{\beta}}^{i,\pm}}, \\
     X_{s_{n-1}}(x_{n-1}) &  if & x_{n-1} \in W_1 \setminus V, \\
     X_{t_{n-1}}(x_{n-1}) &   if  & x_{n-1} \in W_0 \setminus V \textrm{ and } X_{t_{n-1}}(x_{n-1}) \notin V,\\ 
     X_{r_{n-1}}(x_{n-1}) &   if  & x_{n-1} \in W_0 \setminus V \textrm{ and } X_{t_{n-1}}(x_{n-1}) \in V.
             \end{array}
   \right.$$
   
  For a fixed $n\geq 1$, let us consider the following sets
   \begin{itemize}
       \item $O_n=\{0\leq i\leq n  : x_i \in \widetilde{\Sigma_{\widetilde{\beta}}^{i,\pm}}\}$, $n_O= \# O_n$,
       \item $A_n=\{0\leq i\leq n : x_i \in W_1\} $, $n_A= \# A_n$, 
       \item $B_n=\{0\leq i\leq n : x_i \in W_0 \textrm{ and } X_{t_i}(x_i) \notin V \}$, $n_B= \# B_n$,
       \item $C_n=\{0\leq i\leq n : x_i \in W_0\textrm{ and } X_{t_i}(x_i) \in V\} $ and $n_C= \# C_n$.
   \end{itemize}
Note that for every $x_n$, there exists $t'(n)>0$ such that $x_n=X_{t'(n)}(x)$ where

 \begin{equation} \label{tABO}
t'(n)=s_0+ \sum_{i \in A_n}s_i+ \sum_{i \in B_n}t_i+ \sum_{i \in C_n}r_i+ \sum_{i \in O_n}\tau(x_{i}).   \end{equation}
If $y$ satisfies  
\begin{equation}\label{shadowing}
 d(X_t(x),X_t(y))\leq \delta_4, \forall t\in \mathbb{R},
\end{equation} 
define $y_n=X_{t'(n)}(y)$.
 
Let consider $x_n \in \widetilde{\Sigma_{\widetilde{\beta}}^{i,\pm}}$. By hypothesis, we have that 
 \begin{equation*}
  d(x_n,y_n) \leq \delta_4 \textrm{ and } 
  d(X_{\tau (x_n)}(x_n),X_{\tau (x_n)}(y_n)) \leq \delta_4.   
 \end{equation*}
 By using twice Lemma \ref{rL}, there exists $l_0,l_1 \in (-L,L)$ such that
 \begin{equation*}
     y'=X_{l_0}(y_n)\in\widetilde{\Sigma_{\widetilde{\beta}}^{i,\pm}} \textrm{ and }
     X_{\tau (x_n)+l_1}(y_n) \in\widetilde{\Sigma_{\widetilde{\beta}}^{o,\pm}}.
 \end{equation*}
 Thus, 
 \begin{equation} \label{l0l1}
 \tau(x_n)=\tau(y')+l_0-l_1,   
 \end{equation}
 and by (\ref{l0l1}) and lemma (\ref{lema31}) we have that
 \begin{eqnarray} \label{c331}
  \vert\det DX_{\tau(x_n)}(y_n)\vert_{E^c_{y_n}}\vert 
  &\geq&  \vert\det DX_{-l_{1}}(X_{\tau(y')+l_0}(y_n))\vert_{E^c_{X_{\tau(y')+l_0}(y_n)}}\vert \\ 
  \nonumber & & \vert\det DX_{\tau(y')}(y')\vert_{E^c_{y'}}\vert  \vert\det DX_{l_0}(y_n)\vert_{E^c_{y_n}}\vert \\
  \nonumber &\geq& K' C_0^{\tau(x_n)} \vert\det DX_{\tau(x_n)}(x_n)\vert_{E^c_{x_n}}\vert,
 \end{eqnarray} 
 where 
 \begin{equation*} 
  K'=K_0^2K_1 \textrm{ and }  
K_0=\min_{(z,s)\in \widetilde{U}\times [-L,T+L] } \vert\det DX_{s}(z)\vert_{E^c_{z}}\vert. 
 \end{equation*}
Let $\lbrace t_k\rbrace_{k\geq1}$ be an unbounded and increasing sequence of $C$-hyperbolic times for $x$. Up to a slight change of $C$, we can assume that  $X_{t_k}(x) \notin V$. Thus, for every $k\geq 1$ one can write  $t_k=t'_k +u_k$, with $u_k \in [0,T)$ and $t_k'$ of the form (\ref{tABO}). Let us denote $\varphi_t(z)=\vert\det DX_{t}(z)\vert_{E^c_{z}}\vert$.  By the relations (\ref{estimateW}), (\ref{estimateoutside}) and (\ref{c331}), 
 \begin{eqnarray*}
  \vert\det DX_{t_x}(y)\vert_{E^c_y}\vert 
  &=& \vert\det DX_{u}(X_{t'_x}(y))\vert_{E^c_{X_{t'_x}(y)}}\vert
  \vert\det DX_{t'_x}(y) \vert_{E^c_y}\vert  \\
 &\geq & K_0 \prod_{i\in O_n} \varphi_{\tau(x_i)}(y_i) \prod_{i\in A_n} \varphi_{s_i}(y_i) \prod_{i\in B_n} \varphi_{t_i}(y_i) \prod_{i\in C_n} \varphi_{r_i}(y_i) \\
 &\geq & \overline{K} (K')^{n_O} C_0^{\sum_{i\in O_n}\tau(x_i)} K_3^{n_A} K_2^{n_B} \left( \frac{K_0}{K_5}\right)^{n_C+1}\varphi_{t_x}(x),
 \end{eqnarray*} 
 where
$$\overline{K}=\frac{\vert\det DX_{s_0}(y)\vert_{E^c_{y}}\vert}{\vert\det DX_{s_0}(x)\vert_{E^c_{x}}\vert } \textrm{ and } K_5=\max_{(z,s)\in \widetilde{U}\times [-L,T+L] } \vert\det DX_{s}(z)\vert_{E^c_{z}}\vert. $$  

Since $t_x$ is a $C$-hyperbolic time for $x$, by the above estimate we have that
\begin{equation} \label{C-}
\vert\det DX_{t_x}(y)\vert_{E^c_{y}}\vert \geq \overline{K} e^{(C+N(C_0)+N(K')+N(K_2)+N(K_3)+N(K_0,K_5))t_x},    
\end{equation}
 where
 \begin{itemize}
 \item $N(C_0)=\frac{\sum_{i\in O_n}\tau(x_i)}{t_x} \ln(C_0)$,
 \item $N(K')=\frac{n_O}{t_x}\ln(K')$
  \item $N(K_2)=\frac{n_A}{t_x}\ln(K_2(\epsilon_0))$,
  \item $N(K_3)=\frac{n_B}{t_x}\ln(K_3(\epsilon_0))$ and
  \item $N(K_0,K_5)=\frac{n_C+1}{t_x}\ln\left(\frac{K_0}{K_5}\right)$.
  \end{itemize}
  
 Now, since $t_x > \sum_{i\in O_n}\tau(x_i)$ it follows that
 \begin{equation*}
\vert N(C_0) \vert \leq | \ln C_0 |.      
 \end{equation*}
Furthermore, since $n_C+1 \leq n_O$ by the construction of the sequence $x_n$ we have
 \begin{equation*}
     \Big | \frac{n_O}{t_x} \ln K' + \frac{n_C+1}{t_x} \ln \frac{K_0}{K_5} \Big | \leq
      \frac{n_O}{t_x}\Big | \ln K' + \ln \frac{K_0}{K_5} \Big | \leq \frac{|\ln K'| +|\ln \frac{K_0}{K_5}|}{n_0}.
 \end{equation*}
 Then, by shrinking $V_{\widetilde{\beta}}$ if it is necessary, we have 
 \begin{equation*}
       \Big | \frac{n_O}{t_x} \ln K' + \frac{n_C+1}{t_x} \ln \frac{K_0}{K_5} \Big |\leq \frac{|\ln K'| +|\ln \frac{K_0}{K_5}|}{n_0}\leq  \alpha.
 \end{equation*}
 On the other hand, since $\frac{(n_A+n_B+n_C+n_O)}{t_x}\leq 1$, we obtain 

      $$|N(K_2)| \leq |\ln K_2(\varepsilon_0)| \textrm{ and }
     |N(K_3)| \leq |\ln K_3(\varepsilon_1)|, $$
and by (\ref{c*}) and (\ref{C-}) we have
 \begin{eqnarray*}
 \vert\det DX_{t_x}(y)\vert_{E^c_y}\vert&\geq& \overline{K} e^{(C-(\alpha +| \ln C_0|+|\ln K_2(\varepsilon_0)|+|\ln K_3(\varepsilon_1)|))t_x} 
   \geq \overline{K} e^{c^*t_x}.
 \end{eqnarray*}
Finally, there exist $0<c_*<c^*$ and $\overline{T}>0$ such that $ \overline{K} e^{c^*t}\geq e^{c_*t} $ for all $t \geq \overline{T}$.
Therefore, for every $y$ satisfying (\ref{shadowing}) and every $C$-hyperbolic time $t_x \geq \overline{T}$ we have   
 \begin{equation*}
     \vert\det DX_{t_x}(y)\vert_{E^c_y}\vert\geq \overline{K} e^{c^*t_x} \geq e^{c_*t_x},
 \end{equation*}
 and this concludes the proof.
 \end{proof}

Next, we can finally prove the kinematic expansiveness of $\Lambda$. The proof we present here is very similar to the proof of Theorem A of \cite{AP}, as most of the steps from that work can be directly applied here. The only exception is Theorem 3.1 of \cite{AP}, which cannot be replicated directly in our context. Hence, we will be mainly focused on providing a detailed proof of a version of Theorem 3.1 from \cite{AP} adapted to our setting, and we will briefly explain the other steps of the argument.
  
\begin{proof}[Proof of Theorem \ref{k-exp}]

 We begin assuming that $X$  is not kinematic expansive on an ASH attractor $\Lambda$. Then, there are $\varepsilon>0$, $x_n,y_n\in \Lambda$ and $\delta_n\to 0$ such that $y_n\notin X_{[-\varepsilon,\varepsilon]}(x_n)$ and 
 \begin{equation}\label{shad}
     d(X_t(x_n),X_t(y_n))\leq \delta_n \textrm{ for every } t\in \mathbb{R}.
 \end{equation}

Following the arguments in \cite{AP} we can find a regular point $z\in \Lambda$ and $z_n\in \omega(x_n)$ such that $z_n\to z$. Let us consider $\Sigma_{\eta}$ an $\eta$-adapted cross-section trough $z$. By using flow boxes in a small neighborhood of $\Sigma_{\eta}\cup\widetilde{\Sigma^{i,o,\pm}}$ we can  find positive numbers $\delta', t_0$ such that for any $\Sigma'\subset \Sigma_{\eta}\cup\widetilde{\Sigma^{i,o,\pm}}$, $z\in\Sigma'$ and $w\in M$ with $d(z,w)<\delta'$, there is $t_w\leq t_0$ such that $w'=X_{t_w}(w)\in\Sigma'$ and $d_{\Sigma'}(z,w')<K'\delta'$, where $d_{\Sigma'}$ is the intrinsic distance in $\Sigma'$, for some constant positive constant $K'$ which depends on $\Sigma_{\eta}\cup\widetilde{\Sigma^{i,o,\pm}}$.

Let $N>0$ be large enough such that $$0<\delta_N<\min\left\lbrace\delta_4,\eta,\eta_0,\delta'\right\rbrace,$$ where $\eta_0$ and $\delta_4$ are given by Remark 4.2, Lemma \ref{jac}  respectively. Let us denote $x=x_N$  and $y=y_N$. The next claim is our version of Theorem 3.1 in \cite{AP}.

\vspace{0.1in}
\textbf{Claim: }There is $s\in\mathbb{R}$ such that $X_{s}(y)\in W_{\varepsilon}^{ss}(X_{[s-\varepsilon,s+\varepsilon]}(x))$.
\vspace{0.1in}

Once the claim holds, by following the proof of Theorem A in \cite{AP} step by step, one can conclude the proof of the theorem. Because of this, we now devote ourselves to proving the previous claim.
By construction, the orbit of $x$ must intersect $\Sigma_{\eta}$ in infinitely many positive times $t_j$. Let us denote $x_j=X_{t_j}(x)$. Thus, we can find a sequence of times $s_j$ close to $t_j$ such that $y_j=X_{s_j}(y)$ also intersect $\Sigma_{\eta}$. Now, we briefly recall the construction of the tube-like domain presented in \cite{AP}. For any $j\geq 0$ we can find a smooth immersion $$\rho^j:[0,1]\times[0,1]\to M $$ such that the following holds:

\begin{enumerate}
    \item $\rho^j([0,1]\times \{0\})$ is the orbit arc from $x_j$ to $x_{j+1}$ and $\rho^j([0,1]\times \{1\})$ is the orbit arc from $y_j$ to $y_{j+1}$ and 
    \item $\rho^j(\{0\}\times [0,1])$ is a curve contained in $\Sigma_{\eta}$, everywhere tranverse to $W^s(\Sigma_{\eta})$ and joining $x_j$ and $y_j$.
    \item $\rho^j(\{1\}\times [0,1])$ is a curve contained in $\Sigma_{\eta}$, everywhere tranverse to $W^s(\Sigma_{\eta})$ and joining $x_{j+1}$ and $y_{j+1}$.
    \item Denote $S_j=\rho^j([0,1]\times [0,1])$. Then the intersection of  $S_j$ with any $\Sigma^{i,o,\pm}_{\sigma}$ is transverse to stable foliation of $\Sigma^{i,o,\pm}_{\sigma}$.
    \end{enumerate} 
 Denote  $$\mathcal{T}_j=\bigcup_{p\in S_j}W^s_{loc}(p).$$
 The results in \cite{AP} show that $\mathcal{T}_j$ does not contain singularities. Moreover, if $\mathcal{T}_{j}$ intersects some $\Sigma^{i,o,\pm}_{\sigma}$, this intersection is contained in the same connected component of $\Sigma^{i,o,\pm}_{\sigma}\setminus W^s(\sigma)$. Finally, in \cite{AP} it was also shown the existence of a Poincar\'{e} map $R_j$ between the whole strip between the stable manifolds of $x_j$ and $y_j$ inside $\Sigma_{\eta}$. Unfortunately, we cannot guarantee that this Poincar\'{e} maps are hyperbolic in the same way as in \cite{AP} because we are in the ASH setting. Nevertheless, we will obtain some expansion for these maps using the previous lemmas.  To prove the claim we first assume that $$X_{s}(y)\notin W_{\varepsilon}^{ss}(X_{[s-\varepsilon,s+\varepsilon]}(x)),$$ for every $s\in\mathbb{R}$. 
In this case, we have $y\notin\mathcal{O}(x)$; otherwise by Lemma \ref{conj} and Remark \ref{rmk} we obtain that $y$ belongs to  $X_{[-\varepsilon,\varepsilon]}(x)$, and this is a contradiction.

   In addition, unless to take a subsequence of $x_j$, we can find a sequence of arbitrarily large $C$-hyperbolic times $\lbrace t_j\rbrace_{j\geq 1}$ of $x$ such that $x_j=X_{t_j+r_j}(x)\in\Sigma_{\eta}$, where $0\leq r_j<T$, with $$T=\max\lbrace T_0,T_1 \rbrace.$$
   Besides, we have that $y_j=X_{t_j+r_j+v_j}(y)\in\Sigma_{\eta}$, where $\vert v_j\vert\leq t_0$. Moreover, by shrinking $U$ if it is necessary, by Lemma 2.7 in \cite{AP} there is $\kappa_0>0$ such that 
\begin{equation}\label{cotabajo}
\ell(\gamma_n)\leq\kappa_0d_{\Sigma'}(x_j,y_j)\leq \kappa_0K'\delta',    
\end{equation}
where $\gamma_j$ is any curve joining $x_j$ and $y_j$, for every $n\geq 1$. In particular, this holds for the curves $$\gamma_j=\rho^j(\{1\}\times [0,1])\subset \Sigma_{\eta}.$$

On the other hand, take
\begin{displaymath}
\kappa=\min_{(z,s)\in \overline{B_{\delta'}(\Sigma_{\eta})}\times [-t_0,t_0]}\vert\det DX_{s}(z)\vert_{E^c_z}\vert\cdot\min_{(z,s)\in\overline{U}\times [0,T]}\vert\det DX_{s}(z)\vert_{E^c_z}\vert>0.
\end{displaymath}
Let $\lambda>\kappa_0K'\delta'$, and let $j_1$ large enough such that $\kappa e^{c_*t_{j_1}}>\lambda$. Let $R$ be a Poincar\'{e} map whit return time  
\begin{displaymath}
s(x)\approx t_{j_2}+T\quad s(z')\approx t_0+t_{j_2}+T,\quad\forall z'\in\gamma_{0}, 
\end{displaymath}
where $j_2>j_1$ is large enough. Figure 2 helps visualize the definition of $R$. 

\begin{figure}[ht]
\includegraphics[scale=0.4]{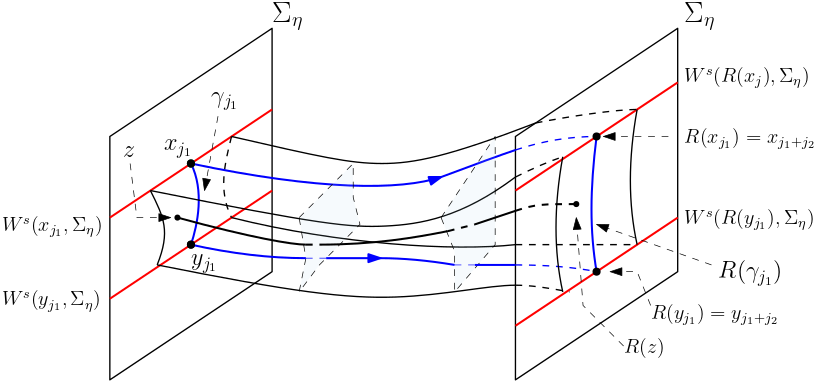}
\caption{The Poincar\'e map $R$.}
\end{figure}

Following the proof of Lemma \ref{jac}, we deduce, by shrinking $\varepsilon'$ if it is necessary, that the relation \eqref{expansion} is satisfied for any $z\in\gamma_{j_1}$. So, by definition of $R$ we have that $R(\gamma_{j_1})$ is a curve in $\Sigma_{\eta}$ that connects $x_{j_1+j_2}$ with $y_{j_1+j_2}$ and satisfies
\begin{displaymath}
\ell(R(\gamma_{j_1}))\geq ke^{c_*t_{n_1}}\geq \lambda> \kappa_0K'\delta',
\end{displaymath}
which contradicts \eqref{cotabajo}. So, we have that $y_{j_1+j_2}\in W^s(x_{j_1+j_2},\Sigma_{\eta})$. Therefore, we obtain the claim by following step by step the argument given in \cite{AP}, p. 2456.  
\end{proof}

\subsection{Positive Entropy and Periodic Orbits}

Next we shall proceed to prove Theorem \ref{HomClass3}. Our argument here is inspired by the work \cite{SMV3}. To begin with,  let us consider $E^s\oplus E^c$ and $U_0$. As a first step for the proof of Theorem \ref{HomClass3}, we show that the ASH property can be extended to the trapping region $U_0$ in the following way: 

\begin{lemma}
Let $\Lambda$ be an asymptotically sectional-hyperbolic attractor associated to a $C^1$ vector field $X$. Then, we have 
\begin{displaymath}
\limsup_{t\to\infty}\frac{1}{t}\log\vert\det DX_t(x)\vert_{E_x^c}\vert>0,\quad\forall x\in U'=U_0\setminus W^s(Sing(x)). 
\end{displaymath}
\end{lemma}
\begin{proof}
First, we need to get an estimate of $\vert\det DX_t(\cdot)\vert_{(\cdot)}\vert$ in a neighborhood of the singularities of $\Lambda$. Let $\sigma$ be either a Rovella-like or resonant singularity, and let $V_{\sigma}$ be a neighborhood of $\sigma$. By shrinking $V_{\sigma}$ if it is necessary, by continuity of $DX_{(\cdot)}(x)$ and the choice of the cone field there are $\theta'_{\sigma}\leq0$ and $T_0>0$ such that for every $x\in V_{\sigma}$, 
\begin{displaymath}
\vert\det DX_t(x)\vert_{L_x}\vert\geq\frac{1}{2}\vert\det DX_t(\sigma)\vert_{E_{\sigma}^c}\vert=\frac{1}{2}e^{\theta'_{\sigma}t},
\end{displaymath}
for every plane $L_x\subset C_a^c(x)$ and every $t\in[0,T_0]$ such that $X_t(x)\in V_{\sigma}$. So, for $x\in V_{\sigma}$ and every $t>0$ such that $X_t(x)\in V_{\sigma}$ we have $t=mT_0+r$, $m\in\mathbb{N}$ and $0\leq r<T_0$, so that, by the above estimation, we have for every plane $L_x\subset C_a^c(x)$ that
\begin{displaymath}
\vert\det DX_t(x)\vert_{L_x}\vert\geq \left(\frac{1}{2}e^{\theta'_{\sigma}r}\right)\left(\frac{1}{2}e^{\theta'_{\sigma}T_0}\right)^m\geq K_{-}(\sigma)e^{\theta_{\sigma}t}, 
\end{displaymath}
where $\theta_{\sigma}=\frac{1}{T_0}\log\left(\frac{1}{2}e^{\theta'_{\sigma}T_0}\right)<0, K_{-}(\sigma)>0$. In a similar way, by taking $T_0$ large enough, there are $K_+>0$ and $\theta_{\sigma}>0$ such that 
\begin{displaymath}
\vert\det DX_t(x)\vert_{L_x}\vert\geq K_+(\sigma)e^{\theta_{\sigma}t},
\end{displaymath}
for every Lorenz-like singularity $\sigma$, every $x\in V_{\sigma}$ such that $X_t(x)\in V_{\sigma}$, and for every plane $L_x\subset C_a^c(x)$. In this case, we set $V_{Sing(X)}=V_R\cup V_L$, where $R$ denotes the set of Rovella-like or resonant singularities and $L$ denotes the set of Lorenz-like singularities. In this case, let
\begin{displaymath}
\theta_+=\min_{\sigma\in L}\theta_{\sigma}>0,\quad\theta_-=\min_{\sigma\in R}\lbrace \theta_{\sigma}\rbrace\leq 0,\text{ and } K=\min_{\sigma\in V_{Sing(X)}}\lbrace K_-(\sigma),K_+(\sigma)\rbrace>0.
\end{displaymath}

On the other hand, by ASH property there is a positive constant $c<C$, where $C$ is given by (\ref{ash}),  and $T_1>T_0$ large enough with the following property: for every $x\in\Lambda'=\Lambda\setminus W^s(Sing(X))$ there is a neighborhood $V_x$ of $x$ such that
\begin{displaymath}
\vert\det DX_t(y)\vert_{L_y}\vert\geq e^{ct_1},\quad \forall y\in V_x, t_1=t_1(x)\geq T_1,
\end{displaymath}
for every plane $L_y\subset C_a^c(y)$, where $t_1$ is the first hyperbolic time for $x$. In this case, denote
\begin{displaymath}
W_x=\bigcup_{0\leq t\leq t_1(x)}X_t(V_x).
\end{displaymath}
Then, the set $W=\bigcup_{x\in\Lambda'}W_x$ defines an open cover of $\Lambda'$. 

By compactness of $\overline{U_0\setminus (V_{Sing(X)}\cup W)}$, there are $T_2>0$ and $b\in\mathbb{R}$ such that
\begin{displaymath}
    \vert \det DX_r(z)\vert_{L_z}\vert\geq e^{br},\quad\forall (z,r)\in \left(\overline{U_0\setminus (V_{Sing(X)}\cup W)}\right)\times[0,T_2],\quad L_z\subset C_a^c(z).   
\end{displaymath}

Now, for every $x\in U_0$, let consider the following numbers:
\begin{itemize}
    \item $\beta_0(x)=\displaystyle\limsup_{t\to+\infty}\frac{1}{t}\int_0^t\chi_{\overline{U_0\setminus (V_{Sing(X)}\cup W)}}(X_s(x))ds$,
    \item $\beta_1(x)=\displaystyle\liminf_{t\to+\infty}\frac{1}{t}\int_0^t\chi_{V_L\cup W}(X_s(x))ds$, and 
    \item $\beta_2(x)=\displaystyle\limsup_{t\to+\infty}\frac{1}{t}\int_0^t\chi_{V_R\setminus W}(X_s(x))ds$.
\end{itemize}
By definition of these numbers, for every $\varepsilon>0$ there is $r>0$ such that for any $t>r$, one has
\begin{displaymath}
  \frac{1}{t}\int_0^t\chi_{\overline{U_0\setminus (V_{Sing(X)}\cup W)}}(X_s(x))ds\leq \beta_0(x)+\varepsilon,\quad \frac{1}{t}\int_0^t\chi_{V_L\cup W}(X_s(x))ds\geq \beta_1(x)-\varepsilon
\end{displaymath}
and
\begin{displaymath}
\frac{1}{t}\int_0^t\chi_{V_R\setminus W}(X_s(x))ds\leq \beta_2(x)+\varepsilon.
\end{displaymath}

So, for any $x\in U_0$, by splitting the orbit into orbit segments and by  joining the above estimations, we have that
\begin{equation}\label{festimate}
\vert \det DX_t(x)\vert_{E^c_x}\vert\geq e^{\psi(x,t)t},\quad\forall t>r,
\end{equation}
with
\begin{displaymath}
\psi(x,t)=\frac{(k_1(t)+k_2(t))\ln K}{t}+b\beta_0(x)+c'\beta_1(x)+\theta_{-}\beta_2(x)+(\theta_--c'+b)\varepsilon.
\end{displaymath}
Here $c'=\min\lbrace c,\theta_+\rbrace>0$ and the numbers $k_1(t)$ and $k_2(t)$ denote how many times the orbit of $x$ hits $V_R$ and $V_L$ in the interval  $[0,t]$, respectively, and these hitting points does not belong to $W$.

Now, let 
\begin{displaymath}
Z=\lbrace x\in U_0 : d(x)=\rho(x)+b\beta_0(x)+c'\beta_1(x)+\theta_-\beta_2(x)\leq 0\rbrace\subset U_0, 
\end{displaymath}
where $\rho(x)=\limsup_{t\to\infty}\left((k_1(t)+k_2(t))\ln K/t\right)<\infty$. Since $\rho(\cdot)$, $\beta_0(\cdot)$, $\beta_1(\cdot)$ and $\beta_2(\cdot)$ are invariant by the flow we have that $Z$ is an  invariant subset of $U_0$. So, since $\Lambda$ is attracting, it follows that $Z\subset\Lambda$. But, in $\Lambda'\cup L$ one has $k_1(t)=0$, $k_2(t)=0\text{ or }1$, $\beta_0(x)=0$, $\beta_2(x)=0$ and $\beta_1(x)=1$, so that $d(x)= c'>0$. In particular, $Z\subset W^s(R)$, so that $d(x)>0$ in $U'$. Moreover, note that $U'=\bigcup_{n\in\mathbb{N}}U'_n$, where $U'_n=\lbrace x\in U' : d(x)>1/n\rbrace$. Therefore, if $x\in U'$ there is $n\in\mathbb{N}$ such that $d(x)>1/n$. Thus, if we choose $\varepsilon_n>0$ satisfying 
\begin{displaymath}
1/n+(\theta_-+b-c')\varepsilon_n>d_n>0,
\end{displaymath} 
we obtain by \eqref{festimate} that
\begin{displaymath}
\limsup_{t\to\infty}\frac{1}{t}\log\vert \det DX_t(x)\vert_{E^c_x}\vert\geq d_n>0. 
\end{displaymath}
This concludes the proof.
\end{proof}

Now, recall that for a continuous map $f:M\to M$, the {\it{empirical probabilities of orbit of a point}} $x\in M$ are defined as
\begin{equation}\label{empiricas}
m_{n,x}=\frac{1}{n}\sum_{i=0}^{n-1}\delta_{f^j(x)},\quad n\in\mathbb{N}, 
\end{equation}  
where $\delta_y$ is the Dirac measure supported at $y\in M$. Let $p\omega_f(x)$ be the set of accumulation points of the sequence \eqref{empiricas} in the weak$^*$ topology. 

\begin{lemma}\label{lemma3.3}
For every point $x \in \Lambda \setminus W^s(Sing(X))$, the set $R$ is not in the support of any measure $\nu\in p\omega_f(x)$, where $f=X_1$.
\end{lemma}

\begin{proof}
Assume that $R\cap supp(\nu)\neq\emptyset$, and let consider $\sigma\in R\cap supp(\nu)$. Take a neighborhood $V$ of $\sigma$. By definition of $R$ and by shrinking $V$ if necessary, there exists $C_0>0$ and $N>0$ such that 
\begin{equation} \label{theta}
\vert\det Df^n(x)\vert_{E^c_x}\vert \leq C_0e^{\underline{\alpha}n}\leq e^{\overline{\alpha}n},    
\end{equation}
where $\underline{\alpha} \leq \overline{\alpha} \leq 0$, for every $n\geq N$ and every $x\in V$  satisfying $f^i(x)\in V$, for $i=0,\ldots,n$. 

Let $x \in \Lambda'$  and let $t_k$ be a sequence of $C$-hyperbolic times for $x$. It is easy to check that one can take  $t_k=n_k\in \mathbb{N}$, so that

\begin{equation} \label{TNH}
   \vert\det Df^{n_k}(x)\vert_{E^c_u}\vert \geq e^{C n_k},\quad k\geq 1. 
\end{equation}
Since $V^c$ is a compact set, there exists $a>1$ such that
\begin{equation} \label{Vc}
    \vert\det Df(z)\vert_{E^c_z}\vert \leq a \ \forall z \in V^c.
\end{equation}
So, if $B_n=\lbrace 1\leq m\leq n : f^m(x)\in V^c\rbrace$, we have by (\ref{theta}), (\ref{TNH}) and  (\ref{Vc}), 
\begin{eqnarray}
 \nonumber a^{\# B_{n_k}} e^{\# B^{c}_{n_k} \overline{\alpha}} &\geq& \prod_{i\in B_{n_k}} \vert\det Df(f^i(x))\vert_{E^c_{f^i(x)}}\vert \prod_{i\in B^c_{n_k}} \vert\det Df(f^i(x))\vert_{E^c_{f^i(x)}}\vert \\
 &=& \prod_{i=0}^{n_k-1} \vert\det Df(f^i(x))\vert_{E^c_{f^i(x)}}\vert \\
\nonumber &=& \vert\det Df^{n_k}(x)\vert_{E^c_x}\vert \\ 
\nonumber &\geq& e^{C n_k}.
\end{eqnarray}
So,
$$ \frac{\# B_{n_k}}{n_k} \geq \frac{\# B_{n_k}}{n_k} + \frac{\# B^c_{n_k}}{n_k \log a}\overline{\alpha} \geq  \frac{C}{\log a} >0, $$
and hence 
$$\limsup_{n\to\infty}\frac{\#B_n}{n}=\alpha(x)>0.$$

Now, by the above relation,
\begin{displaymath}
0\leq\alpha'(x)=\limsup_{n\to\infty}\frac{\#B_n^c}{n}<1.
\end{displaymath}
In particular, if $\varepsilon>0$ satisfies $\alpha'(x)+\varepsilon<\alpha<1$, there is $N\in\mathbb{N}$ such that
\begin{equation}
 \frac{\#B_{n}^c}{n}<\alpha,\quad\forall n\geq N.    
\end{equation}

Let $V'$ be a compact neighborhood of $\sigma$ contained in $V$. By Urysohn's lemma there is a continuous function $\varphi:M\to\mathbb{R}$ such that $\varphi(V')=1$ and $\varphi(M\setminus V)=0$. So, it follows that $\int\varphi d\nu\geq 1$. Therefore, 
\begin{displaymath}
\left\vert \frac{1}{n}\sum_{j=0}^{n-1}\varphi(f^j(x))-\int \varphi d\nu\right\vert\geq 1-\frac{\#B_n^c}{n}>1-\alpha>0,\quad n\geq N. 
\end{displaymath}
This shows that $\nu\notin p\omega_f(x)$, which is a contradiction. This proves the result.
\end{proof}

\begin{proof}[Proof of Theorem \ref{HomClass3}]
Let $f=X_1$. By Lemma \ref{lemma3.3} we have that $supp(\nu)\subset \Lambda'\cup L$, where $\Lambda'=\Lambda\setminus W^s(Sing(X))$, for every $\nu\in p\omega_f(x)$, $x\in U'$. Besides, since $\omega(x)\subset\Lambda$ for any $x\in U'$, we have by Theorem F in \cite{CYZ} that
\begin{displaymath}
    h_{\nu}(f)\geq \int \chi_1(x)+\chi_2(x)d\nu,  
\end{displaymath}
where $\chi_i(x)$, $i=1,2$, are the Lyapunov exponents of $Df$ on $E_x^c$.  

Now, since $\chi_1(\sigma)+\chi_2(\sigma)=\lambda_s(\sigma)+\lambda_u(\sigma)>0$ for every $\sigma\in L$ and  
\begin{equation}\label{sumpos}
\chi_1(x)+\chi_2(x)=\lim_{n\to\infty}\frac{1}{n}\log\vert \det(Df^n(x)\vert_{E_x^c})\vert\geq C,\quad \forall x\in\Lambda',    
\end{equation}
we have 
\begin{displaymath}
h_{\nu}(f)\geq \int_{\Lambda'\cup L} \chi_1(x)+\chi_2(x)d\nu\geq C+ \#(L)\min_{\sigma\in L}[\chi_1(\sigma)+\chi_2(\sigma)]>0.
\end{displaymath}
So, by the Variational Principle, we have $h(X)>0$. So, there is an ergodic measure $\mu$ with $h_{\mu}(X)>0$ supported on $\Lambda'$. Then, by Theorem \ref{erghyp}, the measure $\mu$ is hyperbolic. Moreover, its hyperbolic splitting is induced by the splitting on $\Lambda'$. In particular, this splitting is dominated with index $\text{dim }E$. So, since $\Lambda$ is attracting, it follows from Theorem 4.1 in \cite{PYY} that $\Lambda$ contains a periodic orbit.

For the last part, Let $z\in\Lambda$ such that $\omega(z)=\Lambda$. By hypothesis, there is a sequence of periodic orbits $p_n$, with period $\tau(p_n)$, such that $p_n\to z$. In particular, we have $\tau(p_n)\to\infty$. So, by Hyperbolic lemma, all these orbits are hyperbolic of saddle type, so that there is $N\in\mathbb{N}$ large enough such that $\gamma_{p_n}\sim \gamma_{p_m}$ for $n,m\geq N$, i.e., $W^s(p)\pitchfork W^u(p)$ and $W^s(q)\pitchfork W^u(p)$. This shows that $z\in H(p_N)$, since $H(p_N)$ is closed. Hence, the result is obtained by the denseness of $O(z)$.
\end{proof}

\section*{Acknowledgments}
The first author want to thanks the hospitality of the Department of Mathematics of SUSTech-China where part of this work was developed. The second author wishes to express his sincere gratitude to all the authors, especially the third and fourth, for allowing him to join this project. He learned a great deal from them and greatly values the opportunity to have worked together as a team. He hopes we can continue collaborating in the future. The third author would like to thank the Department of Mathematics of SUSTech-China for their hospitality, where part of this work was developed.

\section*{STATEMENTS}

No data was used for the research described in the article.

\bibliographystyle{plain}

\end{document}